\documentclass[11pt]{amsart}
\usepackage[UKenglish]{babel}
\usepackage[utf8]{inputenc} 
\usepackage{amsthm} 
\usepackage{amsfonts} 
\usepackage{mathtools}
\usepackage{enumitem} \setlist[enumerate]{label={\upshape(\arabic*)}}
\usepackage{amssymb}
\usepackage[numbers]{natbib}
\usepackage{url}
\usepackage{mathrsfs}
\usepackage{geometry}
\usepackage{bbm}
\usepackage{tikz-cd}
\usepackage{tikz}
\usepackage{color}
\usepackage{framed}
\usepackage{hyperref}
\usepackage{todonotes}
\hypersetup{
    colorlinks=true, 
    linkcolor=blue, 
    urlcolor=red, 
    citecolor=[rgb]{0,0.7,0},
    linktoc=all 
}

\theoremstyle{definition}
\newtheorem{defn}{Definition}[section]
\newtheorem{prop}[defn]{Proposition}

\newtheorem{thm}[defn]{Theorem}

\newtheorem{lem}[defn]{Lemma}
\newtheorem{cor}[defn]{Corollary}
\newtheorem{rmk}[defn]{Remark}
\newtheorem{ex}[defn]{Example}

\newtheorem{question}[defn]{Question}

\newcommand{\C}{\mathbb{C}}
\newcommand{\R}{\mathbb{R}}

\newcommand{\Z}{\mathbb{Z}}

\DeclareMathOperator{\MV}{MV}
\DeclareMathOperator{\rk}{rk}
\DeclareMathOperator{\interior}{int}
\DeclareMathOperator{\intr}{int}

\DeclareMathOperator{\conv}{conv}

\DeclareMathOperator{\Gor}{Gor}
\DeclareMathOperator{\Cay}{Cay}
\DeclareMathOperator{\codeg}{codeg}

\DeclareMathOperator{\link}{link}
\def\vol{{\rm vol}}
\newcommand{\pro}[2]{\langle #1, #2 \rangle}

\def\lora{\Longrightarrow}

\title{Thin polytopes: Lattice polytopes with vanishing local $h^*$-polynomial}

\author{Christopher Borger}
 
\author{Andreas Kretschmer}

\author{Benjamin Nill}

\address{Faculty of Mathematics, Otto-von-Guericke-Universit\"at Magdeburg, Universit\"atsplatz 2, 39106 Magdeburg, Germany. Emails: andreas.kretschmer@ovgu.de, benjamin.nill@ovgu.de}


\begin{document}
\maketitle
\setlength{\parindent}{0pt}
\begin{abstract}
In this paper we study the novel notion of thin polytopes: lattice polytopes whose local $h^*$-polynomials vanish. The local $h^*$-polynomial is an important invariant in modern Ehrhart theory. Its definition goes back to Stanley with 
fundamental results achieved by Karu, Borisov \& Mavlyutov, Schepers, and Katz \& Stapledon. The study of thin simplices was originally 
proposed by Gelfand, Kapranov and Zelevinsky, where in this case the local $h^*$-polynomial simply equals its so-called box polynomial. Our main results are the complete classification of thin polytopes up to dimension 3 and the characterization of thinness for Gorenstein polytopes. The paper also includes an introduction to the local $h^*$-polynomial with a survey of previous results.
\end{abstract}

\section{Introduction}

In this paper we propose to investigate {\em thin polytopes}: lattice polytopes with vanishing local $h^*$-polynomials. Local $h^*$-polynomials are also called $\ell^*$-polynomials or $\tilde{S}$-polynomials. In the case of lattice simplices, they equal the so-called box polynomial, see Example~\ref{ex:simplex}. Thin simplices were first defined in the context of regular $A$-determinants and $A$-discri\-minants by Gelfand, Kapranov and Zelevinsky \cite[11.4.B]{GKZ94} as those lattice simplices whose Newton numbers are zero, see Remark~\ref{rem:gkz}. As has been noted in \cite{GKZ94}, ``a classification of thin lattice simplices seems to be an interesting problem in the geometry of numbers.'' In this paper, we extend this endeavor to thin lattice polytopes, which we throughout refer to for simplicity as thin polytopes. Our main results are a complete classification of thin polytopes up to dimension $3$ (Theorem~\ref{thm:3d}) and a characterization of thin Gorenstein polytopes in any dimension (Theorem~\ref{thm:main}). The latter relies crucially on a recent non-negativity result by Katz and Stapledon \cite[Theorem~6.1]{Katz2016Local}. As a consequence, we solve the original problem of \cite{GKZ94} in these two cases and answer questions posed by Borisov, Schepers and the last named author that came up in the investigation of stringy $E$-polynomials of Gorenstein polytopes.
\smallskip

We also hope that this paper leads to renewed interest in the study of the local $h^*$-polynomial as a fundamental invariant of a lattice polytope with many fruitful connections as pioneered in the work of Stanley \cite{Stanley1992Subdivisions}, Karu \cite{Karu08}, Batyrev, Borisov, Mavlyutov \cite{batyrev2008combinatorial, BorisovMavlyutov}, Schepers \cite{Schepers2012Stringy, NS13}, and Katz, Stapledon \cite{Katz2016Local}.

\smallskip

Let us give an overview of this paper. In Section~\ref{sec:local} we give a comprehensive survey on the local $h^*$-polynomial of a lattice polytope. In Section~\ref{sec:thin} we define thin polytopes, present the main examples and discuss several open questions (e.g., Question~\ref{question}). Section~\ref{sec:3d} contains the complete classification of three-dimensional thin polytopes. In particular, we prove that three-dimensional lattice simplices are thin if and only if they are lattice pyramids (Corollary~\ref{cor:tetra}). Section~\ref{sec:gorst} presents the characterization of thin Gorenstein polytopes (Theorem~\ref{thm:main}). In particular, we deduce that thin Gorenstein polytopes have lattice width $1$ (Corollary~\ref{cor:flat-gorst}), being thin is invariant under the duality of Gorenstein polytopes (Corollary~\ref{cor:dual}) and show that Gorenstein simplices are thin if and only if they are lattice pyramids (Corollary~\ref{cor:gorsimplex}). For these results, we study in Section~\ref{sec:join} the behavior of local $h^*$-polynomials under joins, particularly for Gorenstein polytopes.

\section*{Acknowledgment} 
Our deepest thanks go to Jan Schepers whose notes and insights during and after the collaboration with the third author on \cite{NS13} contained in particular Lemma~\ref{lem:scheperscons} and were the basis of the proof of Theorem~\ref{thm:main}. We thank Lev Borisov for his comments, questions and his kindness to share the proof of Proposition~\ref{prop:borisov}. We are also grateful for Sam Payne and Liam Solus for their interest. We thank the anonymous referees for their careful reading and very useful suggestions. All the authors have been supported by the Research Training Group Mathematical Complexity Reduction funded by the Deutsche Forschungsgemeinschaft (DFG, German Research Foundation) - 314838170, GRK 2297 MathCoRe.

\section{A primer on the local $h^*$-polynomial of a lattice polytope}
\label{sec:local}

As the local $h^*$-polynomial is still not as well known in Ehrhart theory as the $h^*$-polynomial and also has been studied with different names and notations, we will give a slightly more thorough account on previous research than strictly necessary for the mere purpose of the results of this paper.

\subsection{Toric $g$- and $h$-polynomials of lower Eulerian posets}

In \cite{Stanley1987Generalized}, Stanley generalized the notion of $h$-vectors of simplicial complexes and simplicial polytopes significantly. For this, let us recall some basic terminology.

\begin{defn}
    The dual of a finite poset $\mathcal{P}$ is denoted $\mathcal{P}^\ast$. A finite poset $\mathcal{P}$ is \emph{locally graded} if every inclusion-maximal chain in every interval $[x,y]$ has the same length $r(x,y)$. The \emph{rank} $\rk(\mathcal{P})$ is the length of the longest chain in $\mathcal{P}$. If in addition there exists a rank function $\rho: \mathcal{P} \rightarrow \Z$, i.e., $r(x,y) = \rho(y) - \rho(x)$ for every interval $[x,y]$, then $\mathcal{P}$ is called \emph{ranked}. If $\mathcal{P}$ is ranked and every interval $[x,y]$ with $x \neq y$ has the same number of even rank and odd rank elements, then $\mathcal{P}$ is \emph{locally Eulerian}. If $\mathcal{P}$ is locally Eulerian and contains a minimal element $\hat{0}$, then it is called \emph{lower Eulerian}. If it also contains a maximum $\hat{1}$, then $\mathcal{P}$ is called \emph{Eulerian}. In presence of a minimum $\hat{0}$ in a ranked poset $\mathcal{P}$, we will always assume that the rank function satisfies $\rho(\hat{0}) = 0$.
\end{defn}

Here is the definition of the $g$-polynomial and the $h$-polynomial for lower Eulerian posets according to Stanley \cite{Stanley1987Generalized}.

\begin{defn}
    Let $\mathcal{P}$ be a lower Eulerian poset with rank function $\rho$ and rank $d$. We define the \emph{$g$-polynomial $g_\mathcal{P}(t)$} and the \emph{$h$-polynomial $h_\mathcal{P}(t)$} recursively by introducing a third polynomial $f_\mathcal{P}(t)$ as an intermediate step. Let
    \begin{gather*}
        f_\emptyset(t) = g_\emptyset(t) = h_\emptyset(t) = 1
        \end{gather*}
        and if $\mathcal{P}\not=\emptyset$, we set 
        \begin{gather*}
        f_\mathcal{P}(t) = \sum_{x \in \mathcal{P}} (t-1)^{d - \rho(x)} g_{[\hat{0},x)}(t)
        \end{gather*}
        and define for $f_{\mathcal{P}}(t)=\sum_{i=0}^df_i t^i$,
        \begin{gather*}
        g_\mathcal{P}(t) =  \sum_{i=0}^{\lfloor d/2 \rfloor} (f_i-f_{i-1}) t^i, \text{ and }\\
        h_\mathcal{P}(t) = \sum_{i=0}^d f_{d-i} t^i.
    \end{gather*}
    \label{def:fgh}
\end{defn}

Hence, $h_\mathcal{P}(t)$ is a polynomial with constant term $1$ of degree $\leq d$, and $g_\mathcal{P}(t)$ is a polynomial with degree $\le d/2$.\\

\begin{rmk} If for $x \in \mathcal{P}$, the interval $[\hat{0},x]$ is boolean, then $g_{[\hat{0},x)}(t) = 1$, see \cite[Prop.~2.1]{Stanley1987Generalized}.\label{simplicial}
\end{rmk}

\begin{rmk} Let us recall the situation of simplicial complexes $\Delta$ (see \cite{Stanley1987Generalized}), where the previous definition of the $h$-polynomial agrees with the usual one. For this, we identify $\Delta$ with its face poset which is a lower Eulerian poset with minimum $\emptyset \in \Delta$. Throughout, we use the convention that $\dim(\emptyset) = -1$. It follows from Remark~\ref{simplicial} that $g_{[\emptyset,\sigma)}(t) = 1$ for all faces $\sigma$ of $\Delta$. If $\Delta$ has dimension $d-1$, we get $f_\Delta(t) = \sum_{i=0}^d f_{i-1} (t-1)^{d-i}$, where $f_j$ denotes the number of faces of $\Delta$ of dimension $j$. Hence, this implies \[h_\Delta(t) = \sum_{\sigma \in \Delta} t^{\dim(\sigma)+1} (1-t)^{d-1-\dim(\sigma)},\] which indeed equals the usual $h$-polynomial of $\Delta$, and where its coefficients form the usual $h$-vector of $\Delta$ \cite[p.~199]{Stanley1987Generalized}. For instance, if $\Delta$ is the boundary complex of a $d$-dimensional simplex, then $h_\Delta(t)=1+t+\cdots+t^d$. Let us also give one example to illustrate that the previous formula for $h_\Delta(t)$ fails in the non-simplicial situation. Let $P$ be the pyramid over the square. In this case, the $h$-polynomial of the boundary complex of $P$ equals $1+2t+2t^2+t^3$, while the previous formula would give $1+2t+t^2+t^3$. Note that the $h$-polynomial is palindromic while the latter expression is not.
\end{rmk}

Stanley proved in \cite[Theorem~2.4]{Stanley1987Generalized} the following combinatorial palindromicity result generalizing the Dehn-Sommerville equations for face numbers of simplicial polytopes.

\begin{thm}
    Let $\hat{\mathcal{P}}$ be an Eulerian poset and $\mathcal{P} \coloneqq \hat{\mathcal{P}} \setminus \hat{1}$ with $\rk(\mathcal{P}) = d$. Then the $h$-polynomial $h_\mathcal{P}(t) = \sum_{i=0}^d h_i t^i$ is palindromic of degree $d$, i.e. $h_i = h_{d-i}$ for all $i=0, \ldots, d$.\label{thm:pali}
\end{thm}

In particular, we have $f_{\mathcal{P}}(t) = h_{\mathcal{P}}(t)$ in this case.

\begin{rmk} We emphasize that in the situation of Theorem~\ref{thm:pali} it is important to distinguish between the $g$- and $h$-polynomials of $\mathcal{P}$ and $\hat{\mathcal{P}}$. Indeed, $g_{\hat{\mathcal{P}}}(t) = 0$ and $h_{\hat{\mathcal{P}}}(t) = g_\mathcal{P}(t)$. Unfortunately, in this regard the different notations employed in the literature can be confusing. Our notation follows that of Stanley while Katz and Stapledon in \cite{Katz2016Local} write $g(\hat{\mathcal{P}};t)$ for our $g_\mathcal{P}(t)$ but also use $h(\mathcal{P};t)$ for our $h_\mathcal{P}(t)$. In Borisov and Mavylutov \cite{BorisovMavlyutov}, as well as in \cite{batyrev2008combinatorial, NS13}, our $g_\mathcal{P}(t)$ and $h_\mathcal{P}(t)$ would be $g_{\hat{\mathcal{P}}^\ast}(t)$ and $h_{\hat{\mathcal{P}}^\ast}(t)$.
        \label{rmk:ambiguous}
    \end{rmk}

Let us give the definition of $h$- and $g$-polynomials of polytopes.

\begin{defn}
 For $P$ a polytope we define its {\em (toric) $h$-polynomial} $h_P(t)$, and its {\em (toric) $g$-polynomial} $g_P(t)$ as the $h$-, resp.,  $g$-polynomial of the face lattice $[\emptyset, P)$ of proper faces of $P$. Note that $g_P(t) = h_{[\emptyset,P]}(t)$, see Remark~\ref{rmk:ambiguous}. 
\end{defn}

Note that by Remark~\ref{simplicial}, we have $g_P(t)=1$ if $P$ is a simplex. 

\begin{thm}\label{thm:h-pol}
    Let $P$ be a polytope of dimension $d$. 
 \label{item:h_pol_unimodal} Then  $h_P(t)=\sum_{i=0}^d h_i t^i$ is a palindromic polynomial with positive integer coefficients that form a unimodal sequence, i.e.,
        \begin{equation*}
            1 = h_0 \leq h_1 \leq \cdots \leq h_{\lfloor \frac{d}{2} \rfloor}.
        \end{equation*}
        Equivalently, $g_P(t)$ has non-negative coefficients.
        \end{thm}
        
        \begin{proof}
Palindromicity follows from Theorem~\ref{thm:pali}. For rational polytopes $P$, nonnegativity follows from the interpretation of the coefficients of $h_P(t)$ as the dimensions of the even intersection cohomology groups of the toric variety associated with $P$ and the unimodality property follows from the hard Lefschetz theorem \cite[Theorem~3.1, Corollary~3.2]{Stanley1987Generalized}. The non-rational case has been treated by Karu in  \cite{Karu04Hard}.
\end{proof}

Let us mention the following less well-known duality property of $g$-polynomials that will be of importance in Section~\ref{sec:gorst}. This is a result by Kalai, published in \cite[Theorem~4.5]{BradenKalai} as a consequence of the main result in that paper by Braden. Here, $\mathcal{P}^*$ denotes the {\em dual} poset of a poset $\mathcal{P}$. 

\begin{thm}
    \label{thm:Kalai} Let $P$ be a polytope. Then \begin{equation*}
            \deg(g_{[\emptyset, P)}) = \deg(g_{(\emptyset,P]^\ast}).
        \end{equation*}
        In other words, if $Q$ is any polytope which is combinatorially dual to $P$, then $\deg(g_P) = \deg(g_Q)$.
\end{thm}

\subsection{(Relative) local $h$-polynomials of polyhedral subdivisions}

We give the definition of the local $h$-polynomial (and its generalized relative version) of a polyhedral subdivision $\Delta$ of a polytope $P$, following \cite{Stanley1992Subdivisions} and \cite{Katz2016Local} (the relative version was introduced independently in \cite{Athanasiadis} and \cite{Report}). Here, we define the {\em link} of a face $\sigma \in \Delta$ as
$\link(\Delta,\sigma) := \{\rho \in \Delta \,:\, \sigma \subseteq \rho \}$. We view $\link(\Delta,\sigma)$ as a lower Eulerian poset with minimum~$\sigma$. 

\begin{defn}
  Let $P$ be a polytope, $\Delta$ a polyhedral subdivision of $P$, and $\sigma \in \Delta$. The \emph{relative local $h$-polynomial} of $\Delta$ with respect to $\sigma$ is defined as
    \begin{equation*}
        \ell_{\Delta,\sigma}(t) \coloneqq \sum_{\sigma \subseteq F \leq P} (-1)^{\dim(P) - \dim(F)}   h_{\link(\Delta_F,\sigma)}(t) \, g_{(F,P]^\ast}(t),
    \end{equation*}
    where $F \le P$ means that $F$ is a face of $P$ (including $\emptyset$ and $P$) and $\Delta_F := \{\rho \in \Delta \colon \rho \subseteq F\}$. 
    
    We call $\ell_\Delta (t) := \ell_{\Delta,\emptyset}$ the \emph{local $h$-polynomial} of $\Delta$.
    \label{def:local-h}
\end{defn}

We suppress $P$ in this notation as it equals $|\Delta| = \bigcup_{\sigma \in \Delta} \sigma$, the support of $\Delta$. We remark that the same definition of the local $h$-polynomial can be extended to so called strong formal subdivisions of Eulerian posets, see \cite[Definition~4.1]{Katz2016Local}.

\begin{thm}\label{thm:local_h}
    Let $P$ be a polytope of dimension $d$, $\Delta$ a polyhedral subdivision of $P$, and $\sigma \in \Delta$. Then we can write $\ell_{\Delta,\sigma}(t) = \sum_{i=0}^{d-\dim(\sigma)} \ell_i t^i$. Moreover, the following holds:
    \begin{enumerate}
        \item $\ell_{\Delta,\sigma}(t)$ has nonnegative integer coefficients.
                \item \label{item:local_h_symmetric} $\ell_{\Delta,\sigma}(t)$ is palindromic, i.e. $\ell_i = \ell_{d-\dim(\sigma)-i}$ for $i=0, \ldots, d-\dim(\sigma)$. 
                \item \label{item:local_h_unimodal} If $\Delta$ is regular, then the coefficients of $\ell_{\Delta,\sigma}(t)$ form a unimodal sequence.
    \end{enumerate}
\end{thm}

\begin{proof} Let us give the references. 
(2): For the local $h$-polynomial this is a special case of \cite[Corollary~7.7]{Stanley1992Subdivisions}, for the relative local $h$-polynomial see \cite[Corollary~4.5]{Katz2016Local}.
(1) and (3): For $\Delta$ a rational polyhedral subdivision this has been proven in \cite[Theorem~7.9]{Stanley1992Subdivisions}, respectively, \cite[Theorem~6.1]{Katz2016Local} using the decomposition theorem (cf. \cite{Beilinson, Decomposition, Propermaps}). 
As pointed out in \cite[Remark~6.6]{Katz2016Local}, the only missing ingredient to drop the rationality hypothesis was the relative hard Lefschetz theorem for the intersection cohomology of fans which was subsequently proven in \cite{Karu2019Relative}.
\end{proof}

The following decomposition theorem was one of the main motivations of Stanley for the notion of {\em local} $h$-vectors. This is proven in \cite[Theorem~7.8]{Stanley1992Subdivisions}, and the general version in \cite{Katz2016Local} (see, e.g., second equation in proof of Lemma~6.4). To stress the analogy to Theorem~\ref{thm:bm}, we state the equality also using the $h$-polynomial.

\begin{prop}
 Let $P$ be a polytope of dimension $d$, $\Delta$ a polyhedral subdivision of $P$, and $\sigma \in \Delta$. 
 Then 
 \[h_{\link(\Delta,\sigma)}(t) = \sum_{\sigma \subseteq F \le P} \ell_{\Delta_F,\sigma}(t) g_{[F,P)}(t) = \sum_{\sigma \subseteq F \le P} \ell_{\Delta_F,\sigma}(t) h_{[F,P]}(t).\]
  In particular for $\sigma=\emptyset$, we get $\ell_\Delta(t) \le h_\Delta(t)$ and $g_P(t)=  h_{[\emptyset,P]} \le h_\Delta(t)$ coefficientwise.
  \label{prop:h-decomp}
\end{prop}

As a consequence of the above nonnegativity results, Stanley and later Katz and Stapledon show that $h$-polynomials as well as relative local $h$-polynomials of polyhedral subdivisions are nonnegative and coefficientwise monotone under subdivision refinement \cite[Corollary~6.10]{Katz2016Local}.

\subsection{The $h^*$-polynomial of a lattice polytope}
We quickly recall the basic notions of Ehrhart theory. Let $P \subseteq \R^d$ be a lattice polytope with respect to some lattice $M \subseteq \R^d$ of maximal rank $d$, usually $M=\Z^d$. The \emph{Ehrhart series of $P$} (with respect to $M$) is the formal power series
    \begin{equation*}
        \mathrm{Ehr}_P(t) \coloneqq 1 + \sum_{n \geq 1} |(nP) \cap M| t^n \in \Z[[t]].
    \end{equation*}
By a theorem of Ehrhart \cite{Ehr62}, the map $\mathrm{ehr}_P \colon \Z_{\geq 1} \rightarrow \Z, \ n \mapsto |(nP) \cap M|$ is a polynomial in $n$, called the \emph{Ehrhart polynomial of $P$}. It has degree $\dim(P)$, constant term $1$ and leading coefficient equal to the volume of $P$ normalized with respect to $M$. It follows that
\begin{equation*}
    \mathrm{Ehr}_P(t) = \frac{h^\ast_P(t)}{(1-t)^{\dim(P) + 1}},
\end{equation*}
for a unique polynomial $h^\ast_P(t) \in \Z[t]$ of degree at most $\dim(P)$, called the \emph{$h^\ast$-polynomial of $P$}. Here, $h^\ast_P(t)$ has non-negative integer coefficients by \cite{Sta80}. Moreover, $h^\ast_P(0) = 1$ and $h^\ast_P(1)$ equals the {\em lattice volume} $\vol_\Z(P) \in \Z_{\ge 1}$, which is defined as $\dim(P)!$ times the volume of $P$ normalized with respect to $M$. Note that the lattice volume of a lattice point equals~$1$. The \emph{degree of $P$}, denoted $\deg(P)$, is the degree of its $h^\ast$-polynomial $h^\ast_P(t)$. The \emph{codegree of $P$}, denoted $\codeg(P)$, is the smallest integer $k \geq 1$ such that the $k$-th dilate $kP$ of $P$ contains a lattice point of $M$ in its relative interior. By convention, a point has codegree~$1$. It follows from Ehrhart-MacDonald reciprocity \cite{Macdonald} that
\begin{equation*}
    \deg(P) + \codeg(P) = \dim(P) + 1.
\end{equation*}

Let us recall that two lattice polytopes $P$ and $Q$ (with respect to the lattice $M$) are called {\em isomorphic} (or {\em unimodularly equivalent}) if there is an affine lattice automorphism of $M$ that maps the vertices of $P$ to the vertices of $Q$. We say $P$ is a {\em unimodular simplex} if $P$ is isomorphic to the convex hull of an affine lattice basis of $M$. Now, $P$ is a unimodular simplex if and only if $h^*_P(t)=1$, or equivalently, $\deg(P)=0$. For $M=\Z^d$ let us also define the {\em standard unimodular simplex} $\Delta_d := \conv(0,e_1, \ldots, e_d)$ for the standard lattice basis $e_1, \ldots, e_d$.

\subsection{The local $h^*$-polynomial of a lattice polytope}

Let us introduce the main player of this paper (see \cite[Example~7.13]{Stanley1992Subdivisions} and \cite[Def.~7.2]{Katz2016Local}).

\begin{defn}
    Let $P$ be a lattice polytope. The \emph{local $h^\ast$-polynomial} or \emph{$\ell^\ast$-polynomial} of $P$ is defined as
    \begin{equation*}
        \ell^\ast_P(t) \coloneqq \sum_{\emptyset \leq F \leq P} (-1)^{\dim(P)-\dim(F)} h^\ast_F(t) g_{(F,P]^\ast}(t).
    \end{equation*}
    \label{def:lstar}
\end{defn}

Let us note that the local $h^*$-polynomial of the empty face equals $1$, while for a point it equals $0$. We also emphasize the analogy of Definition~\ref{def:lstar} with Definition~\ref{def:local-h} above. See also Subsection~\ref{subsec:decompositions} for precise relationships between the $h,\ell,h^\ast,\ell^\ast$-polynomials.

\begin{rmk}
The local $h^*$-polynomial has been studied by Borisov, Batyrev, Mavlyutov, Schepers, and the third author under the name {\em $\tilde{S}$-polynomial}, see \cite[Definition~5.3]{BorisovMavlyutov}. It was used by Borisov and Mavlyutov to simplify the formulas for the stringy $E$-polynomial of Calabi-Yau complete intersections in Gorenstein toric Fano varieties originally described via so-called $B$-polynomials \cite{BB96a}. We remark that the reader should be aware that in these papers in the definition of $h$- and $g$-polynomials the dual poset was used compared to the one given here. 
\end{rmk}

\begin{ex}
For lattice simplices $P$ (of dimension $d>0$) the $h^*$- and $\ell^*$-polynomial can be easily understood, as in this case the face posets are all Boolean. Let $\Pi$ denote the half-open parallelepiped spanned by the vertices of $P \times \{1\}$. Then $h^*_P(t)$ (resp., $\ell^*_P(t)$) enumerates the number of lattice points in $\Pi$ (resp., in the interior of $\Pi$). More precisely, we have $h^*_P(t)=\sum_{i=0}^{d+1} h^*_i t^i$ and $\ell^*_P(t)=\sum_{i=0}^{d+1} \ell^*_i t^i$, where for $i=0, \ldots, d+1$ the coefficient $h^*_i$ (resp. $\ell^*_i$) equals the number of lattice points in $\Pi$ (resp., in the interior of $\Pi$) with last coordinate $i$. We refer to \cite[Prop. 4.6]{batyrev2008combinatorial}. This polynomial $\ell^*_P(t)$ of a lattice simplex $P$ is also often called {\em box polynomial}, cf. \cite{Braun-Survey,Liam, Liam-Gustafsson}. For instance, we have $h^*_P(t)=1$ if and only if $P$ is a unimodular simplex; in this case, $\ell^*_P(t)=0$. Let us note that for $h^*$-polynomials this combinatorial interpretation of its coefficients can also be extended to arbitrary lattice polytopes, e.g., by half-open decompositions \cite{Halfopen}. On the other hand, there is not yet a combinatorial counting interpretation for the coefficients of the local $h^*$-polynomial of lattice polytopes known. 
\label{ex:simplex}
\end{ex}

Let us summarize some of the basic properties of the local $h^*$-polynomial. Throughout, we use the convention that the degree of the zero-polynomial is zero.

\begin{thm}
        Let $P$ be a lattice polytope of dimension $d > 0$. Then we can write the local $h^*$-polynomial $\ell^*_P(t) = \sum_{i=1}^{d} \ell^*_i t^i$. Moreover, the following holds:
    \begin{enumerate}
\item $\ell^*_P(t)$ has nonnegative integer coefficients.
\item $\ell^*_P(t)$ is palindromic: we have $\ell^*_i = \ell^*_{d+1-i}$ for $i=1, \ldots, d$. 
\item If $\ell^*_P(t)$ does not vanish, then the degree of $\ell^*_P(t)$ equals at most the degree of $h^*_P(t)$, and its subdegree (i.e., the smallest $i$ such that the $i$th coefficient of $\ell^*(t)$ is nonzero) is at least the codegree of $P$.
\item $\ell^*_1 = \ell^*_d$ equals the number of lattice points in the interior of $P$.
\end{enumerate}
\label{thm:lstar}
\end{thm}

\begin{proof}
Let us give the corresponding references: 
(1) This was conjectured by Stanley \cite[Conj.7.14]{Stanley1992Subdivisions} and proven by Karu \cite{Karu08}. Using the $\tilde{S}$-notation for $\ell^*$ it also follows from its interpretation as the Hilbert function of a graded vector space by Borisov, Mavlyutov \cite[Prop.~5.5]{BorisovMavlyutov}. (2) This was observed in \cite[Remark 5.4]{BorisovMavlyutov}. (3) This follows directly, see also \cite[Cor.2.16(2)]{NS13}. (4) For this observation, see \cite[Example~4.7]{batyrev2008combinatorial}.
\end{proof}

In particular, the number $\intr_\Z(P)$ of interior lattice points completely determines the local $h^*$-polynomial up to dimension $2$. If $d=0$, then $\ell^*_P(t)=0$; if $d=1$, then $\ell^*_P(t) = \intr_\Z(P) t$; and if $d=2$, then $\ell^*_P(t) = \intr_\Z(P) t + \intr_\Z(P) t^2$.

\subsection{Decomposing and relating the $h,\ell,h^*,\ell^*$-polynomials}\label{subsec:decompositions}

The following classical result by Betke and McMullen, generalized by Katz and Stapledon, explains the relation of $h^*$-polynomials to $h$-polynomials of a {\em lattice} subdivision (i.e., a polyhedral subdivision whose vertices are lattice points). Recall that a lattice triangulation is called {\em unimodular} if all its simplices are unimodular simplices. 

\begin{thm}
    Let $P$ be a lattice polytope with a lattice subdivision $\Delta$. Then the following holds:
    \[h^*_P(t) = \sum_{\sigma \in \Delta}\, \ell^*_\sigma(t)\, h_{\link(\Delta,\sigma)}(t).\] 
    In particular, we have $h_\Delta(t) \le h^*_P(t)$ coefficientwise, where we have equality if and only if the local $h^*$-polynomial of every non-empty face of $\Delta$ vanishes. If $\Delta$ is a lattice triangulation, then this is equivalent to $\Delta$ being a unimodular triangulation.
    \label{thm:bm}
\end{thm}

\begin{proof}

This is Lemma~7.12(3) of \cite{Katz2016Local}, generalizing \cite{BetkeMcMullen85}. We recall that the consequence follows from the nonnegativity of the occuring polynomials and the fact that the $h$-polynomials have constant term $1$. Second, the combinatorial description of the $h^*$- and $\ell^*$-polynomial of a lattice simplex, Example~\ref{ex:simplex}, implies that a lattice simplex $S$ is a unimodular simplex if and only if $\ell^*_\sigma(t) = 0$ for all non-empty faces $\sigma$ of $S$. 
\end{proof}

Let us note that this result was one motivation for Stanley to define local $h$-polynomials, as these allowed him to prove an analogous result in the combinatorial setting, namely, Proposition~\ref{prop:h-decomp} above. And similar to that formula positively expressing the $h$-polynomial of a subdivision into local $h$-polynomials and toric $h$-polynomials, one can also decompose the $h^*$-polynomial of a lattice polytope positively into local $h^*$-polynomials and toric $h$-polynomials of its face poset.

\begin{cor}
Let $P$ be a lattice polytope. Then
\[h^*_P(t) = \sum_{\emptyset \leq F \leq P} \ell^\ast_F(t) g_{[F,P)}(t)=\sum_{\emptyset \leq F \leq P} \ell^\ast_F(t) h_{[F,P]}(t).\]
In particular, $\ell^*_P(t) + g_P(t) = \ell^*_P(t) + h_{[\emptyset,P]}(t) \le h^*_P(t)$ coefficientwise.
\label{cor:schepers}
\end{cor}

A proof in greater generality is given in \cite[Prop.~2.9]{Schepers2012Stringy}, see also \cite[Cor.~1.1]{Karu08} and \cite[Prop.~2.5]{NS13}. We remark that Corollary~\ref{cor:schepers} gives significance to thinking of the local $h^*$-polynomial as the "Ehrhart core" of the $h^*$-polynomial. This is most prominently clear in the case of lattice simplices, see Example~\ref{ex:simplex}.

Now, just as the (generalized) Betke-McMullen formula transparently separates the lattice data (the $\ell^*$-polynomials of the cells) and combinatorial data (the $h$-polynomials of the links of the cells) of the $h^*$-polynomial of the support of a lattice subdivision, the same can be done for the local $h^*$-polynomial. This was observed in \cite{Report}, see also \cite[Lemma~7.12(4)]{Katz2016Local}.

\begin{prop}
Let $P$ be a lattice polytope with a lattice subdivision $\Delta$. Then the following holds:
\[\ell^*_P(t) = \sum_{\sigma \in \Delta} \ell^*_\sigma(t)\; \ell_{\Delta,\sigma}(t).\]
In particular, $\ell_\Delta(t) \le \ell^*_P(t)$ coefficientwise, with equality if $\Delta$ is a unimodular triangulation.
\label{prop:decomp}
\end{prop}

Here, we critically used the nonnegativity of the relative local $h$-polynomial,~Theorem~\ref{thm:local_h}(1), for the consequence. In particular, we get another proof of the nonnegativity of the local $h^*$-polynomial. Moreover, as already observed in \cite{Report},this implies that the unimodality of the $\ell^*$-vector is an {\em intrinsic} obstruction for a lattice polytope to have a unimodular triangulation (apply Theorem~\ref{thm:local_h}(3) with $\sigma=\emptyset$). In fact, it is enough to have unimodality of the "local box polynomials" (this is Remark~7.23 in \cite{Katz2016Local}). Such triangulations where called {\em box unimodal} in \cite{unimodal}.

\begin{cor}
If $P$ admits a regular triangulation such that the local $h^*$-polynomials of each cell have unimodal coefficients (e.g., the triangulation is unimodular), then its local $h^*$-polynomial has unimodal coefficients.\label{cor:unimodcoeffs}
\end{cor}

This was used in \cite{Liam-Gustafsson} to prove the unimodality of the local $h^*$-polynomial of $s$-lecture hall order polytopes.

\smallskip

For the purpose of this paper, the following innocent looking consequence of the nonnegativity of relative local $h$-polynomials is crucial. 

\begin{cor}
Let $P$ and $P'$ be lattice polytopes such that $P'$ is obtained from $P$ by refining the lattice. Then $h^*_P(t)\le h^*_{P'}(t)$ and $\ell^*_P(t) \le \ell^*_{P'}(t)$ coefficientwise.
\label{cor:lattice}
\end{cor}

\begin{proof}

This follows from Theorem~\ref{thm:bm}, respectively, Proposition~\ref{prop:decomp}, the explicit combinatorial description of the $\ell^*$-polynomial of a lattice simplex, see Example~\ref{ex:simplex}, and the nonnegativity of the $h$-polynomial, respectively, of the relative local $h$-polynomial.
\end{proof}

We remark that for $h^*$-polynomials this lattice-monotonicity can also easily be seen combinatorially, e.g., using half-open decompositions (see \cite{reciprocity}). However, for local $h^*$-polynomials there seems to be no such combinatorial argument known. This is also true for the next result. Note that by Stanley's famous monotonicity result, the $h^*$-polynomial is coefficientwise monotone with respect to inclusion. However, this is not true for the local $h^*$-polynomial. Still, it holds when one considers subpolytopes that do not lie on the boundary.

\begin{cor}
Let $P$ and $Q$ be lattice polytopes such that ${\rm relint}(Q) \subseteq \intr(P)$ (for instance, $\dim(Q)=\dim(P)$). Then 
$\ell^*_Q(t) \le \ell^*_{P}(t)$ coefficientwise.\label{cor:mono}
\end{cor}

\begin{proof}
Choose a lattice subdivision $\Delta$ of $P$ that contains $Q$ as a cell. Then by the nonnegativity of the appearing polynomials, we see from Proposition~\ref{prop:decomp} that $\ell^*_Q(t) \ell_{\Delta,Q}(t) \le \ell^*_P(t)$ coefficientwise. It remains to observe that since $Q$ is in the relative interior of $P$, it follows directly from Definition~\ref{def:local-h} that $\ell_{\Delta,Q}(t) = h_{\link(\Delta,Q)}(t)$ and hence has constant coefficient~$1$.
\end{proof}

Finally, let us just shortly mention that in the Katz-Stapledon paper \cite{Katz2016Local}, motivated by algebraic and tropical geometry, the $h^*$- and $\ell^*$-polynomials are further refined to bivariate (and even trivariate) polynomials leading to the notion of $h^*$- and $\ell^*$-diamonds. As we will use the following notation later for the computation of the local $h^*$-polynomial in dimension three, let us introduce it here. 

\begin{defn}
Let $P$ be a lattice polytope of dimension $d$. Then we define
\[h^*_P(u,v) := \sum_{F \le P}v^{\dim(F)+1} \ell^*_F(uv^{-1}) g_{[F,P)}(uv).\]
\label{def:mixed}
\end{defn}

Note that $h^*_P(t,1)=h^*_P(t)$. In \cite[Remark~7.7]{Katz2016Local} it is shown that
\[h^*_P(u,v) = 1 + uv \sum_{0 \le p,q \le d-1} h^*_{p,q} u^p v^q,\]
where $h^*_{i,d-1-i} = \ell^*_i$ for $i=1,\ldots, d$. These refined invariants satisfy many beautiful properties. Let us present here at least one such consequence, namely, the following lower bound theorem on the coefficients of the $\ell^*$-polynomial \cite[p.184]{Katz2016Local}.

\begin{thm}
    Let $P$ be a lattice polytope of dimension $d$ with $\ell^*_P(t)=\sum_{i=1}^d \ell^*_i t^*$. Then 
    \[\ell^*_1 \le \ell^*_i \quad \text{ for } i=2, \ldots, d.\]
    \label{thm:lower}
\end{thm}

\section{Definition, basic properties and examples of thin polytopes}
\label{sec:thin}

\subsection{Main definition and known results from \cite{GKZ94}}

The following notion is the main focus of this paper. 

\begin{defn} A lattice polytope $P$ is called {\em thin} if its local $h^*$-polynomial $\ell^*_P$ vanishes. By the nonnegativity of the coefficients, Theorem~\ref{thm:lstar}(1), this is equivalent to $\ell^*_P(1)=0$.
\end{defn}

Let us note that lattice polytopes of dimension $0$ as well as unimodular simplices are thin, see Example~\ref{ex:simplex}. We remark that thin polytopes naturally appear in Theorem~\ref{thm:bm}.

\begin{rmk}
Thin simplices were first investigated in \cite[11-4-B]{GKZ94} in the context of regular $A$-determinants and $A$-discriminants, more precisely, in the characterization of so-called $D$-equivalence classes of regular triangulations of $A$. There a lattice simplex $S$ was defined to be {\em thin} if its {\em Newton number} $\nu(S)$ equals zero. Here, the Newton number is defined as follows:
\begin{equation}
    \nu(S) := \sum_{\emptyset \le F \le S} (-1)^{\dim(S)-\dim(F)} \vol_\Z(F) =0,
    \label{newton}
    \end{equation}
where, $\vol_\Z(\emptyset):=1$ (also $\vol_\Z(F)=1$ if $\dim(F)=0$). Recall from Example~\ref{ex:simplex}, that $\vol_\Z(F)=h^*_F(1)$ counts the number of lattice points in the half-open parallelepiped over $F$. Hence, by inclusion-exclusion, it is straightforward to deduce $\nu(S)=\ell^*_S(1)$, the number of interior lattice points in the half-open parallelepiped over $S$. Thus, for lattice simplices the definitions agree. Let us note that in \cite{GKZ94} the nonnegativity of $\nu(S)$ follows from quite deep algebro-geometric arguments, while it is combinatorially obvious from the interpretation of $\ell^*_S$ as the box polynomial of the lattice simplex $S$.\label{rem:gkz} The reader should also be warned that the expression in equation~\eqref{newton} may be negative for lattice {\em polytopes}. For instance, it equals $-1$ for the 0/1-cube $[0,1]^3$.
\end{rmk}

Thin simplices were classified in \cite{GKZ94} up to dimension $2$. Here, we can deduce the following statement directly from Theorem~\ref{thm:lstar}(4). Let us define a lattice polytope to be {\em hollow} if it has no lattice points in its interior. Here, a $0$-dimensional lattice polytope is not hollow (but thin).

\begin{prop}
Thin polytopes of dimension $> 0$ are hollow. The converse also holds in dimensions $1$ and $2$.
\label{prop:dim2}
\end{prop}

In particular, $\Delta_1$ is the only thin polytope of dimension $1$. Hollow polytopes in dimension $2$ are well-known. They are either isomorphic to $2 \Delta_2$ or have lattice width $1$ (i.e., all vertices lie on two parallel hyperplanes of lattice distance one). Note that hollow three-dimensional lattice polytopes do not have to be thin., e.g., $2 \Delta_3$ and $[0,1]^3$ are not thin.

One important construction for thin polytopes is to take lattice pyramids.

\begin{defn}
Let $P \subset \R^d$ be a lattice polytope. Then 
\[\conv(P \times \{0\}, \{0\} \times \{1\}) \subset \R^d \times \R\]
is called the {\em lattice pyramid} over $P$. By convention, a lattice point is also considered a lattice pyramid.
\end{defn}

It is well-known that the $h^*$-polynomial, and particularly the degree, does not change under taking lattice pyramids. The following result has already been observed in \cite{GKZ94} for lattice simplices and in general in \cite{batyrev2008combinatorial} for lattice polytopes.

\begin{prop}
Lattice pyramids over arbitrary lattice polytopes are thin.\label{prop:pyr}
\end{prop}

Using this notation we can state the classification of thin simplices up to dimension two as follows.

\begin{cor}
A lattice simplex of dimension at most $\le 2$ is thin if and only if it is isomorphic to $2 \Delta_2$ or it is a lattice pyramid.
\end{cor}

\begin{rmk}
In \cite{GKZ94} {\em thin triangulations} were intensively studied. Recently, this notion has also been investigated by \cite{Thintriangulations} where it was completely characterized up to dimension $3$. As Stanley observed at the end of Section~7 in \cite{Stanley1992Subdivisions}, a thin triangulation may be defined by the vanishing of its local $h$-polynomial. Now, it follows from Proposition~\ref{prop:decomp} that {\em all} lattice triangulations of thin polytopes are thin. This seems to be a quite strong combinatorial obstruction worth of further study.
\end{rmk}

\begin{rmk}
    By Corollary~\ref{cor:lattice}, a thin polytope stays thin if the lattice is coarsened. We do not know of a purely combinatorial proof of this fact.\label{rmk:refine}
\end{rmk}

\begin{rmk}
    If a lattice polytope is contained in a thin polytope but not in its boundary, then it is also thin. This non-trivial fact follows from Corollary~\ref{cor:mono}.
\end{rmk}

\begin{rmk}
    Let us note that thin simplices turn up in \cite{unimodal} when studying conditions for unimodality of (local) $h^*$-polynomials in the context of box unimodal triangulations mentioned before Corollary~\ref{cor:unimodcoeffs}. Here, let us recall the following observation: if $P$ admits a regular triangulation $\Delta$ such that every non-empty face of $\Delta$ is thin, then its local $h^*$-polynomial equals the local $h$-polynomial of $\Delta$ and its $h^*$-polynomial equals the $h$-polynomial of $\Delta$, see Proposition~\ref{prop:decomp} and Theorem~\ref{thm:bm}. Now, in \cite{unimodal} it is asked whether every IDP lattice polytope has a regular triangulation into lattice simplices that have vanishing or monomial $\ell^*$-polynomial. The motivation was that the existence of a box unimodal triangulation of an IDP reflexive polytope implies unimodality of its $h^*$-polynomial. While the previous question is still open, a proof of the latter result using completely different methods was recently announced in \cite{Karim}.
\end{rmk}

\subsection{Two classes of examples of thin polytopes}

\label{subsec:two}
Let us describe two ways to get thin polytopes in higher dimensions.

The first observation is that lattice polytopes of small degree (in other words, "very hollow" lattice polytopes) are always thin.

\begin{defn}
We say, $P$ is {\em trivially thin} if $\dim(P) \ge 2 \deg(P)$.
\end{defn}

\begin{prop}
Trivially thin polytopes are thin.
\end{prop}

\begin{proof}
A lattice polytope $P$ is trivially thin if and only if $\deg(P) < \codeg(P)$. Now, the statement follows from Theorem~\ref{thm:lstar}(3).
\end{proof}

Typical examples of trivially thin polytopes are Cayley polytopes with many factors. We will talk about Cayley polytopes with two factors in much more detail later (see Definition~\ref{def:cayley} and Remark~\ref{rem:cayley}), however, let us already now give the definition of a Cayley polytope to make the previous statement precise. For this, we denote by a {\em lattice projection} $\R^d \to \R^m$ an affine-linear map mapping $\Z^d$ surjectively onto $\Z^m$. If there is a lattice projection mapping a $d$-dimensional lattice polytope $P$ onto a unimodular simplex $\Delta_k$ with $k \ge 1$, then $P$ is called a {\em Cayley polytope} with $k+1$ factors (namely, the fibers of the vertices of $\Delta_k$). One can easily deduce from \cite[Proposition~1.12]{batyrev2008combinatorial} that $P$ is trivially thin if $k \ge d/2$. An alternative way to view this is also the following. Take $r$ lattice polytopes $P_0, \ldots, P_{r-1}$ in $\R^m$. Then the {\em Cayley sum} of $P_0, \ldots, P_{r-1}$ is defined as the convex hull of $P_0 \times \{0\}$ and $P_i \times \{e_i\}$ for $i=1, \ldots, r-1$ in $\R^{m+r-1}$. It is trivially thin if $r \ge m+1$. Note that a Cayley sum is a Cayley polytope, and every Cayley polytope is isomorphic to a Cayley sum.

\medskip
A second way to get high-dimensional thin polytopes is to use free joins. 

\begin{defn}
Let $P \subset \R^n$ and $Q \subset \R^m$ be lattice polytopes. We call 
\[P \circ_\Z Q := \conv(P \times \{0\} \times \{0\}, \{0\} \times P \times \{1\}) \subset \R^n \times \R^m \times \R,\]
the {\em free join} of $P$ and $Q$.
\end{defn}

For instance, the free join of $[0,1]$ with itself is a unimodular $3$-simplex. Note that isomorphic factors lead to isomorphic free joins. From the Ehrhart-theoretic viewpoint the free join construction is important because of the following multiplicativity property, see \cite[Lemma~1.3]{Join} and \cite[Remark~4.6(5)]{NS13}.

\begin{prop}
Let $P \subset \R^n$ and $Q \subset \R^m$ be lattice polytopes. Then \[h^*_{P \circ_\Z Q}(t) = h^*_P(t) \, h^*_Q(t),\; \text{ and }\; \ell^*_{P \circ_\Z Q}(t) = \ell^*_P(t) \, \ell^*_Q(t).\]
\label{prop:mult}
\end{prop}

\begin{cor}
The free join of two lattice polytopes is thin if and only if at least one of the two factors is thin.
\end{cor}

As a lattice pyramid is the free join of a point (which is thin) and a lattice polytope, this generalizes Proposition~\ref{prop:pyr}.

\subsection{Are there other examples of thin polytopes?}
\label{subsec:ex}
It is not trivial to give examples of thin polytopes (such as Example~\ref{ex:nonsp} below) that do not fall in above described two classes. In order to formulate a natural question in this respect, let us recall two notions. First, a lattice polytope $P$ is called {\em spanning} if every lattice point in its affine hull is an integer affine combination of the lattice points in $P$. Note that every lattice polytope becomes spanning after a possible coarsening of the ambient lattice (we refer to \cite{spanning} for more background and results on spanning lattice polytopes). Second, let us call a lattice polytope $P$ a {\em join} if there are two non-empty faces $F$ and $G$ of $P$ such that the free join of $F$ and $G$ is affinely-isomorphic to $P$. Let us remark that if $P$ is spanning and a join of $F$ and $G$ where every lattice point in $P$ is contained in $F$ or $G$, then it is automatically a free join.

\begin{question}\label{question}\
\begin{enumerate}
    \item Is every thin polytope trivially thin or a join?
    \item Is every spanning thin polytope trivially thin or a free join?
    \end{enumerate}
\end{question}

Both questions are closely related but not directly. The reason is that the degree of the polytope can drop under coarsenings of the lattice, so a non-spanning thin but not trivially thin polytope could be trivially thin with respect to its spanning lattice. We also note that trivially thin polytopes are often not joins. For example the unit square $[0,1]^2$ has degree $1$ and is hence trivially thin, while triangles are the only polygons which are joins. 

As the following example shows, the spanning hypothesis in the second part of Question~\ref{question} is indeed important. It is one of the apparently rare thin polytopes that are not trivially thin and not a free join.

\begin{ex}
Consider the $4$-simplex $P = \conv(0, e_1, e_2, (1, 2, 4, 0), (2, 1, 0, 4)) \subseteq \R^4$. The sublattice $N$ of $\Z^4$ spanned by all lattice points of $P$ has index $2$ and the quotient $\Z^4/N \cong \Z/2\Z$ is generated by $\overline{e_3} = \overline{e_4}$. A computation in \texttt{SageMath} with backend \texttt{Normaliz} shows that $P$ is thin and $h_P^\ast(t) = t^3 + 11t^2 + 3t + 1$, in particular $\deg(P) = 3$, so $P$ is not trivially thin. A computation in \texttt{Polymake} shows that the lattice width of $P$ is $2$, so that $P$ is not a Cayley polytope, in particular not a free join. It can be checked that with respect to $N$, $P$ is the lattice pyramid over a reflexive $3$-simplex of lattice volume $8$.
\label{ex:nonsp}
\end{ex}

Question~\ref{question} should be understood as a guiding question for finding interesting high-dimensional thin polytopes. Let us discuss this problem in more detail below.

As being hollow is equivalent to $\deg(P) < \dim(P)$, it is evident that every hollow lattice polytope in dimension $\le 2$ is trivially thin. Hence, by Proposition~\ref{prop:dim2} every thin polytope in dimension $\le 2$ is trivially thin. It will be proven in our first main result Theorem~\ref{thm:3d} that in dimension $3$ all non-trivially thin polytopes are lattice pyramids. In particular, Question~\ref{question} has an affirmative answer in dimensions $\le 3$. Note that $\conv(e_1,e_2,-e_1-e_2) \circ_\Z 2 \Delta_2$ is an example of a thin simplex in dimension $5$ that is not trivially thin (it has degree $3$), but is not a lattice pyramid, while being a free join with a (trivially) thin factor.

In higher dimensions our second main result shows that non-trivially thin Gorenstein polytopes are so-called Gorenstein joins (see Definition~\ref{def:Gorenstein-join}) with a trivially thin factor, so that Question~\ref{question} has an affirmative answer also in the Gorenstein case (see Corollary~\ref{cor:spanning}).

Computationally, we have verified that Question~\ref{question} has an affirmative answer for all $4$-dimensional lattice polytopes of lattice volume $\leq 21$, for all $5$-dimensional lattice simplices of lattice volume $\leq 20$ and for all $6$-dimensional lattice simplices of lattice volume $\leq 16$. We provide some of the relevant data at \cite{github}.

\subsection{Interesting thin empty simplices?}

 A lattice simplex is called {\em empty} if its vertices are its only lattice points. 
Among the hollow polytopes this is the class of lattice simplices that has been studied most intensively, see e.g. \cite{Santos4d} and the references therein. However, it turns out that there are no interesting thin empty simplices in dimension at most $4$. Let us give the easy reasoning. For this, we recall that the {\em quotient group} of a $d$-dimensional lattice simplex $P \subset \R^d$ is defined as the quotient of $\Z^{d+1}$ by the subgroup generated by the vertices of $P \times \{1\}$. 

\begin{prop}
Let $P$ be a lattice simplex with cyclic quotient group. Then $P$ is thin if and only if $P$ is a lattice pyramid.
\end{prop}

\begin{proof}
Let $P \subset \R^d$ be $d$-dimensional. We denote by $\Pi$ the half-open parallelepiped from Example~\ref{ex:simplex}. Clearly, every element in the quotient group of $P$ has a unique representative in $\Pi \cap \Z^{d+1}$. Let $g \in \Pi \cap \Z^{d+1}$ be the representative of a generator of the quotient group of $P$. We assume that $P$ is thin. Hence, there is a proper, non-empty subset $V'$ of the vertex set of $S \times \{1\}$ such that $g$ is a linear combination of vertices of $V'$. In particular, this also holds for the representatives of all the elements in the quotient group of $P$. Now, it follows from \cite[Lemma~12]{NillSimplices} that $P$ is a lattice pyramid.
\end{proof}

It is well-known that all empty lattice simplices in dimension at most $4$ have cyclic quotient group \cite{barile}. As also in higher dimensions most  empty simplices constructed (but not all of them) have this property (see e.g. \cite{WideSimplices}), it seems to be a challenge to find examples of empty simplices that are thin but not simply lattice pyramids.

\subsection{Are thin polytopes `flat'?}

We observed above that all thin polytopes in dimension at most two have lattice width $1$ except for $2 \Delta_2$. We leave it as an exercise to the reader to show that $2 \Delta_d$ for even $d$ is the only thin simplex among all lattice simplices of the form $\conv(0,k_1 e_1, \ldots, k_d e_d) \subset \R^d$ with $k_1, \ldots, k_d \in \Z_{\ge 1}$ that are not lattice pyramids (i.e., $k_i > 1$ for all $i$). It will follow from our main results that thin polytopes in dimension three (Corollary~\ref{cor:flat-3d}) as well as thin Gorenstein polytopes in arbitrary dimension (Corollary~\ref{cor:flat-gorst}) have lattice width $1$. In dimension four Example~\ref{ex:nonsp} has lattice width $2$. As thin polytopes (of dimension $> 0$) are hollow, in fixed dimension their lattice width is bounded. Now, our lack of 'non-flat' examples motivates the following question.

\begin{question}\label{question2}
Can one find (spanning) thin polytopes of arbitrarily large lattice width?
\end{question}

We expect that such examples with increasing lattice width should exist with increasing dimension. Note that if one assumes that Question~\ref{question}(2) has an affirmative answer, then for Question~\ref{question2} it would be important to find the maximum width of trivially thin spanning polytopes $P$. However, it is a folklore open question, often called 'the' Cayley conjecture (see \cite{Alicia,HNP09,Higashitani}), that any lattice polytope with $\dim(P) > 2 \deg(P)$ has lattice width $1$. Thus, assuming also that the Cayley conjecture holds essentially reduces the previous question to the study of spanning lattice polytopes with $\dim(P) = 2 \deg(P)$. 

\section{Classification of thin polytopes in dimension $3$}

\label{sec:3d}

As observed above, three-dimensional lattice polytopes $P$ that are lattice pyramids over polygons or have degree at most one are automatically thin. Our first main result, Theorem~\ref{thm:3d}, shows that in dimension three indeed all the thin polytopes are of this type. 

Lattice polytopes of degree at most one are completely known in any dimension. For this, let us recall the following definition.

\begin{defn}
A {\em Lawrence prism} is a $d$-dimensional lattice polytope in $\R^d$ isomorphic to $\conv(0,e_1,\ldots, e_{d-1}, k_0 e_d, e_1+k_1 e_d, \ldots, e_{d-1}+k_{d-1} e_d)$ with $k_0, k_1, \ldots, k_{d-1} \in \Z_{\ge 1}$.
\end{defn}

The following result was proven in \cite{batyrev2007multiples}.

\begin{thm}
Any lattice polytope of degree $1$ is either a lattice pyramid, a Lawrence prism or isomorphic to $2 \Delta_2$.
\label{thm:BN}
\end{thm}

Here is the main result of this section.

\begin{thm}
	Let $P$ be a $3$-dimensional lattice polytope. Then $P$ is thin if and only if $P$ is a lattice pyramid over a lattice polygon (i.e., over a lattice polytope of dimension~$2$) or $\deg(P) \le 1$. Equivalently, $P$ is thin if and only if 
	\begin{itemize}
				\item $P$ is a lattice pyramid over a lattice polygon, or
				\item $P$ is a Lawrence prism.
	\end{itemize}
	\label{thm:3d}
\end{thm}

\begin{cor}
Every three-dimensional thin polytope has lattice width $1$.\label{cor:flat-3d}
\end{cor}

The proof of Theorem~\ref{thm:3d} relies on two instances that seem to be exceptional to small dimensions. First, in dimension three all the coefficients of the local $h^*$-polynomial can be explicitly determined.

\begin{prop}
	Let $P \subseteq \R^3$ be a $3$-dimensional lattice polytope. Then
	\begin{equation*}
		\ell^\ast_P(t) = |\interior_{\Z}(P)| (t + t^3) + \left(|\interior_{\Z}(2P)| - 4 |\interior_{\Z}(P)| - \sum_{F \leq P \text{ facet}} |\interior_{\Z}(F)|\right) t^2.
	\end{equation*}
	\label{prop:lstar-3d}
\end{prop}
	
	\begin{proof}
	Recall from Theorem~\ref{thm:local_h} that $\ell^\ast_1 = \ell^\ast_{3} = h^\ast_{3} = |\interior_{\Z}(P)|$. Hence, we need only determine $\ell^\ast_2$. From Stanley reciprocity, we deduce $h^\ast_2=|\interior_{\Z}(2P)| - 4 |\interior_{\Z}(P)|$. Now, in the notation of the $h^\ast$-diamond introduced in \cite{Katz2016Local} (see Definition~\ref{def:mixed}) we have $h^\ast_2 = h^\ast_{1,0}+h^\ast_{1,1}$, where $h^\ast_{1,1}=\ell^\ast_2$ and $h^\ast_{1,0} = \sum_{F \leq P \text{ facet}} |\interior_{\Z}(F)|$ by \cite[Example~8.9]{Katz2016Local}. This implies the statement.
	\end{proof}
	
	Let us note that we get from the lower bound theorem of Katz-Stapledon, Theorem~\ref{thm:lower}, $\ell^*_1 \le \ell^*_2$. This leads to the following non-obvious corollary. It would be very interesting to find a purely combinatorial proof. 
	
	\begin{cor}
	Let $P \subseteq \R^3$ be a $3$-dimensional lattice polytope. Then 
	\[|\interior_{\Z}(2P)| \ge 5 \,|\interior_{\Z}(P)| + \sum_{F \leq P \text{ facet}} |\interior_{\Z}(F)|.\]
	\label{cor:comb3d}
	\end{cor}
	
	For our purposes, let us note the following numerical characterization of thinness in dimension three.
	
\begin{cor} \label{cor:thinCondition}
	Let $P \subseteq \R^3$ be a $3$-dimensional lattice polytope. Then  $P$ is thin if and only if $P$ is hollow and
	\begin{equation*}
		|\interior_{\Z}(2P)| = \sum_{F \leq P \text{ facet}} |\interior_{\Z}(F)|.
	\end{equation*}
\end{cor}

The second result that is not yet available in higher dimensions is a complete classification of hollow three-dimensional lattice polytopes. 

\begin{thm}[\cite{averkov2011maximal}]\label{thm:3Dhollow}
	Let $P \subseteq \R^3$ be a $3$-dimensional hollow lattice polytope. Then one of the following holds:
	\begin{enumerate}
		\item \label{it:exceptions} $P$ is contained in one of the $12$ maximal hollow lattice polytopes classified in \cite{averkov2011maximal}.
		\item \label{it:width1} There is a lattice projection $\R^3 \rightarrow \R^1$ mapping $P$ onto $\Delta_1$.
		\item \label{it:2delta2} There is a lattice projection $\R^3 \rightarrow \R^2$ mapping $P$ onto $2 \Delta_2$.
	\end{enumerate}
\end{thm}
	
Before giving the proof of Theorem~\ref{thm:3d} let us also recall the following well-known formula for the mixed volume (e.g. \cite[Corollary~3.2]{nill2020mixed}):

\begin{lem}\label{lem:MVinterior}
	Let $P_1, P_2 \subseteq \R^2$ be lattice polytopes. Then
	\begin{align*}
		\MV(P_1, P_2) = 1 &+ (-1)^{\dim(P_1 + P_2)}|\interior_{\Z}(P_1 + P_2)| \\
		&+(-1)^{\dim(P_1) - 1}|\interior_{\Z}(P_1)| + (-1)^{\dim(P_2) - 1}|\interior_{\Z}(P_2)|.
	\end{align*}
\end{lem}

\begin{proof}[Proof of Theorem~\ref{thm:3d}]
By Corollary~\ref{cor:thinCondition}, $P$ is hollow. We treat the three cases of Theorem~\ref{thm:3Dhollow} separately. A direct computation in \texttt{Magma} deals with case~\ref{it:exceptions}, see \cite{github}.

For case~\ref{it:width1}, denote by $P_1, P_2 \subseteq \R^2$ the preimages in $P$ of the vertices of $\Delta_1$. Note that $P_1$ and $P_2$ are faces of $P$ such that every lattice point of $P$ is either contained in $P_1$ or $P_2$. (We remark that $P$ is a Cayley polytope of $P_1$ and $P_2$ in the notation of Definition~\ref{def:cayley}.) We denote by $\interior_{\Z}^2(Q)$ the set of lattice points in the \emph{absolute} interior of a lattice polytope $Q \subseteq \R^2$. Then
\begin{align*}
	\sum_{F \leq P \text{ facets}} |\interior_{\Z}(F)| &= |\interior_{\Z}^2(P_1)| + |\interior_{\Z}^2(P_2)|, \\
	|\interior_{\Z}(2P)| &= |\interior_{\Z}(P_1 + P_2)|,
\end{align*}
where the second equation follows from the so-called Cayley trick. Therefore, Corollary~\ref{cor:thinCondition} translates into
\begin{equation*}
	|\interior_{\Z}^2(P_1)| + |\interior_{\Z}^2(P_2)| = |\interior_{\Z}(P_1 + P_2)|.
\end{equation*}
In case $\dim(P_1) = \dim(P_2) = 2$, plugging this into Lemma~\ref{lem:MVinterior} yields $\MV(P_1, P_2) = 1$, thus $(P_1, P_2) \cong (\Delta_2, \Delta_2)$ by \cite[Proposition~2.7]{cattani2013mixed}, hence $\deg(P) = 1$.
	
If $\dim(P_1)=2$ and $\dim(P_2)=1$, then Lemma~\ref{lem:MVinterior} yields $\MV(P_1,P_2) = 1 + |\interior_{\Z}(P_2)|$. On the other hand, $\MV(P_1,P_2) = V(\pi_{P_2}(P_1))(|\interior_{\Z}(P_2)|+1)$ by \cite[Theorem~5.3.1]{schneider_2013}, where $\pi_{P_2}$ is a lattice projection along the line segment $P_2$ and $V(\pi_{P_2}(P_1))$ denotes the lattice volume.
Hence, $V(\pi_{P_2}(P_1)) = 1$ and therefore $\pi_{P_2}(P_1) \cong \Delta_1$. The lattice projection of $P$ along $P_2$ is then a lattice projection onto $\Delta_2$. Thus, $\codeg(P) \ge 3$ and hence $\deg(P) \leq 1$, so either $\deg(P) = 1$ or $P \cong \Delta_3$ is a lattice pyramid.
	
The case $\dim(P_1) = \dim(P_2) = 1$ is similar. Lemma~\ref{lem:MVinterior} yields $\MV(P_1,P_2)= 1 + |\interior_{\Z}(P_1)| + |\interior_{\Z}(P_2)|$. On the other hand, again $\MV(P_1,P_2)=V(\pi_{P_2}(P_1))	(|\interior_{\Z}(P_2)|+1)$. We may assume $|\interior_{\Z}(P_1)| \leq |\interior_{\Z}(P_2)|$. If $V(\pi_{P_2}(P_1)) \geq 2$ or $V(\pi_{P_2}(P_1)) = 0$, we obtain a contradiction, so $V(\pi_{P_2}(P_1)) = 1$. The same argument as above shows $\deg(P) \leq 1$.

Finally, if one of the $P_i$ is zero-dimensional, then $P$ is a lattice pyramid.\vspace{1em}

It is left to study case~\ref{it:2delta2}, and we may assume $P$ to be of lattice width at least $2$ because width $1$ is equivalent to $P$ being a Cayley polytope which is precisely case \ref{it:width1}.
	
We distinguish several cases and always start by showing how, in each case, we can associate to each lattice point in the interior of a facet of $P$, in an injective way, a lattice point in the interior of $2P$. We then prove that there always exists an additional lattice point in the interior of $2P$, therefore showing that case~\ref{it:2delta2} does not yield any new thin polytopes by Corollary~\ref{cor:thinCondition}.
	
We may assume that $P$ projects onto $2 \Delta_2$ along the $z$-axis. As lattice projections map interior lattice points to interior lattice points, all interior lattice points of a facet of $P$ are of the form $x_1^a=(1,0,a)$, $x_2^a=(0,1,a)$, or $x_3^a=(1,1,a)$ for suitable $a \in \Z$. By fixing vertices $v_1=(0,2,\alpha)$, $v_2=(2,0,\beta)$, $v_3=(0,0,\gamma)$ of $P$ we hence obtain points $\frac{1}{2}(x_i^a + v_i) \in \interior(P)$, and therefore $(x_i^a + v_i) \in \interior_{\Z}(2P)$ for all $a \in \Z$ such that $x_i^a$ is an interior point of a facet of $P$. Then $(x_i^a+v_i) \neq (x_j^b+v_j)$ if~$i \neq j$ or~$a \neq b$.

Now we show the existence of an additional interior lattice point. Indeed, $P$ can have at most three facets containing interior lattice points, namely at most those facets, if there are such, that project to one of the three edges of $2 \Delta_2$. 

We proceed by distinguishing these different cases. If there is no such facet at all, then Corollary~\ref{cor:thinCondition} implies that $2P$ is hollow, so $\deg(P) \leq 1$, contradicting the fact that there is no lattice polytope of degree $\leq 1$ with width $>1$ by~Theorem~\ref{thm:BN}.
	
Next, assume that $P$ has exactly two facets containing interior lattice points, and say these are the facets opposite to $v_1$ and $v_2$. We pick two such points $(0,1,q)$, $(1,0,r)$. Then we obtain as many interior lattice points in $2P$ of the form $x_1^a+v_1$ or $x_2^a+v_2$ as there are points in the interiors of facets of $P$, and $(1,1,q+r)$ is an additional interior lattice point of $2P$.
	
Next, assume $P$ has three facets containing interior lattice points.	If there exists $i \in \{1,2,3\}$ such that the fiber of $2P$ containing $v_i$ contains more than one lattice point, then we can similarly construct an additional interior lattice point of $2P$ by considering the two points of minimal and maximal height in this fiber. Therefore, we may assume $v_1, v_2, v_3$ to be the unique lattice points of $P$ over $(0,2)$, $(2,0)$ and $(0,0)$, respectively.
This implies that all three facets $F_1$, $F_2$, $F_3$ of $P$ lying over the three edges of $2 \Delta_2$ have a special form. E.g., the facet projecting to the edge $[(0,0), (2,0)]$ is a quadrangle or triangle which, after a unimodular equivalence, looks similar to the following:
	
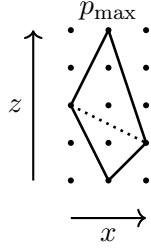
\begin{center}
		\begin{tikzpicture}[scale=0.5,baseline=-5pt] 
		\draw[->, line width=0.8pt] (-1,-2) -- (-1,2) node[midway, left]{$z$};\draw[->, line width=0.8pt] (0,-3) -- (2,-3) node[midway, below]{$x$};\draw[line width=1pt] (1, -2) -- (2, -1);\draw[line width=1pt] (1, 2) -- (2, -1);\draw[line width=1pt] (0, 0) -- (1, 2);\draw[line width=1pt] (0, 0) -- (1, -2);\draw[line width=1pt, dotted] (0, 0) -- (2, -1);\node at (1, -1/4) {}; \node at (1, 0) {}; \node[above] at (1,2) {$p_{\mathrm{max}}$};
		\fill (0,0) circle (0.08);\fill (1,-2) circle (0.08);\fill (1,-1) circle (0.08);\fill (1,0) circle (0.08);\fill (1,1) circle (0.08);\fill (1,2) circle (0.08);\fill (2,-1) circle (0.08);\fill (0,1) circle (0.08);\fill (0,2) circle (0.08);\fill (0,-1) circle (0.08);\fill (0,-2) circle (0.08);\fill (2,0) circle (0.08);\fill (2,-2) circle (0.08);\fill (2,1) circle (0.08);\fill (2,2) circle (0.08);
		\end{tikzpicture}
\end{center}
Now, we fix one interior lattice point $u_i \coloneqq x_i^{a_i}$ in each of the facets for some suitable $a_i \in \Z$.
The three maps $x_i^a \mapsto x_i^a + u_j$ for $(i,j) \in \{(1,2),(2,3),(3,1)\}$ map each interior lattice point of the three facets $F_i$ injectively to an interior lattice point of $2P$. Moreover, the images of these three maps are disjoint by construction.
Now, we may assume that in the facet projecting to the edge $[(0,0), (2,0)]$, the lattice point $p_{\mathrm{max}}$ lying over $(1,0)$ with maximal third coordinate is a vertex as in the picture. But then we obtain the additional interior lattice point $p_{\mathrm{max}}+u_2$ of $2P$ not covered by the images of the three maps above.\vspace{1em}
	
The only remaining case is the one where only one facet $F$ of $P$ contains interior lattice points. We may assume $F$ is the facet which projects onto the edge $[(0,0), (2,0)]$. Then all interior lattice points of $F$ project to $(1,0)$ and are of the form $x_1^a$ for $a \in \Z$ ranging in a suitable interval. From this we obtain $|\interior_{\Z}(F)|$ interior lattice points $v_1 + x_1^a$ of $2P$. Observe that all of them have second coordinate $2$. Therefore, it is enough to show that $2P$ contains an interior lattice point with second coordinate $1$.

Again we proceed by distinguishing cases. If $F$ contains at least three lattice points projecting to $(1,0)$, then it contains the convex hull of $(1,0,a)$, $(1,0, a+1)$, $(1,0,a+2)$, $(2,0,b)$ for some $a,b \in \Z$. Recall that $v_1=(0,2,\alpha)$. Then the point
\begin{equation*}
	2 \cdot \left(\frac{1}{4} v_1 + \frac{1}{4} (1,0,a) + \frac{1}{4} (2,0,b) + \frac{1}{4} (1,0,a+n)\right) = (2,1,a + \frac{\alpha + b + n}{2})
\end{equation*}is an interior lattice point of $2P$ with second coordinate $1$ for precisely one choice of $n \in \{1, 2\}$. Hence, by Corollary~\ref{cor:thinCondition}, $P$ is not thin if $F$ contains at least three lattice points over $(1,0)$, which in particular includes the case $|\interior_{\Z}(F)| \geq 3$. 

Before we proceed further let us observe that the fibers of $P$ over the points $(0,0)$ and $(2,0)$ both consist of at most three lattice points because otherwise another facet than $F$ would contain interior lattice points. This is because, up to unimodular equivalence, there are exactly two lattice pyramids of height $2$ over a lattice segment of length $\geq 3$, and both contain an interior lattice point:
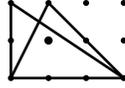
\begin{center}
		\begin{tikzpicture}[scale=0.5,baseline=-5pt] 
		\draw[line width=1pt] (0, 2) -- (3, 0);\draw[line width=1pt] (0, 0) -- (3, 0);\draw[line width=1pt] (0, 0) -- (0, 2);\draw[line width=1pt] (0, 0) -- (1, 2);\draw[line width=1pt] (1, 2) -- (3, 0);\node at (1, 2/3) {}; 
		\fill (0,0) circle (0.08);\fill (1,0) circle (0.08);\fill (2,0) circle (0.08);\fill (3,0) circle (0.08);\fill (0,1) circle (0.08);\fill (1,1) circle (0.11);\fill (0,2) circle (0.08);\fill (1,2) circle (0.08);\fill (2,1) circle (0.08);\fill (2,2) circle (0.08);\fill (3,1) circle (0.08);\fill (3,2) circle (0.08);
		\end{tikzpicture}
\end{center}

Now, we can deal with the remaining case $|\interior_{\Z}(F)| \in \{1,2\}$. Let $F' \subseteq F$ run through the inclusion-minimal subpolygons of $F$ which contain the same interior lattice points as $F$. We will show that there are only two possibilities for $F'$.

Let us first consider the case $|\interior_{\Z}(F)| = 2$. By the previous argument, after a suitable shear fixing the $x = 1$ line (within the $x$-$z$-plane, i.e. $y=0$), $F'$ fits inside the following box:

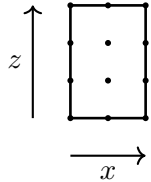
\begin{center}
		\begin{tikzpicture}[scale=0.5,baseline=-5pt] 
		\draw[->, line width=0.8pt] (-1,0) -- (-1,3) node[midway, left]{$z$};\draw[->, line width=0.8pt] (0,-1) -- (2,-1) node[midway, below]{$x$};\draw[line width=1pt] (0, 3) -- (2, 3);\draw[line width=1pt] (2, 0) -- (2, 3);\draw[line width=1pt] (0, 0) -- (2, 0);\draw[line width=1pt] (0, 0) -- (0, 3);\node at (1, 3/2) {}; 
		\fill (0,0) circle (0.08);\fill (0,1) circle (0.08);\fill (0,2) circle (0.08);\fill (0,3) circle (0.08);\fill (1,0) circle (0.08);\fill (1,1) circle (0.08);\fill (1,2) circle (0.08);\fill (1,3) circle (0.08);\fill (2,0) circle (0.08);\fill (2,1) circle (0.08);\fill (2,2) circle (0.08);\fill (2,3) circle (0.08);
		\end{tikzpicture}
\end{center}

As $F'$ is inclusion-minimal, it is isomorphic to one of the following polygons:

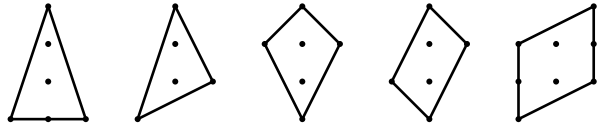
\begin{center}
		\begin{tikzpicture}[scale=0.5,baseline=-5pt] 
		\draw[line width=1pt] (1, 3) -- (2, 0);\draw[line width=1pt] (0, 0) -- (2, 0);\draw[line width=1pt] (0, 0) -- (1, 3);\node at (1, 1) {}; 
		\fill (0,0) circle (0.08);\fill (1,0) circle (0.08);\fill (1,1) circle (0.08);\fill (1,2) circle (0.08);\fill (1,3) circle (0.08);\fill (2,0) circle (0.08);
		\end{tikzpicture}
		\hspace{10pt}
		\begin{tikzpicture}[scale=0.5,baseline=-5pt] 
		\draw[line width=1pt] (1, 3) -- (2, 1);\draw[line width=1pt] (0, 0) -- (2, 1);\draw[line width=1pt] (0, 0) -- (1, 3);\node at (1, 4/3) {}; 
		\fill (0,0) circle (0.08);\fill (1,1) circle (0.08);\fill (1,2) circle (0.08);\fill (1,3) circle (0.08);\fill (2,1) circle (0.08);
		\end{tikzpicture}
		\hspace{10pt}
		\begin{tikzpicture}[scale=0.5,baseline=-5pt] 
		\draw[line width=1pt] (1, 0) -- (2, 2);\draw[line width=1pt] (1, 3) -- (2, 2);\draw[line width=1pt] (0, 2) -- (1, 3);\draw[line width=1pt] (0, 2) -- (1, 0);\node at (1, 7/4) {}; 
		\fill (0,2) circle (0.08);\fill (1,0) circle (0.08);\fill (1,1) circle (0.08);\fill (1,2) circle (0.08);\fill (1,3) circle (0.08);\fill (2,2) circle (0.08);
		\end{tikzpicture}
		\hspace{10pt}
		\begin{tikzpicture}[scale=0.5,baseline=-5pt] 
		\draw[line width=1pt] (1, 0) -- (2, 2);\draw[line width=1pt] (1, 3) -- (2, 2);\draw[line width=1pt] (0, 1) -- (1, 3);\draw[line width=1pt] (0, 1) -- (1, 0);\node at (1, 3/2) {}; 
		\fill (0,1) circle (0.08);\fill (1,0) circle (0.08);\fill (1,1) circle (0.08);\fill (1,2) circle (0.08);\fill (1,3) circle (0.08);\fill (2,2) circle (0.08);
		\end{tikzpicture}
		\hspace{10pt}
		\begin{tikzpicture}[scale=0.5,baseline=-5pt] 
		\draw[line width=1pt] (0, 2) -- (2, 3);\draw[line width=1pt] (2, 1) -- (2, 3);\draw[line width=1pt] (0, 0) -- (2, 1);\draw[line width=1pt] (0, 0) -- (0, 2);\node at (1, 3/2) {}; 
		\fill (0,0) circle (0.08);\fill (0,1) circle (0.08);\fill (0,2) circle (0.08);\fill (1,1) circle (0.08);\fill (1,2) circle (0.08);\fill (2,1) circle (0.08);\fill (2,2) circle (0.08);\fill (2,3) circle (0.08);
		\end{tikzpicture}
\end{center}

But only the last one does not contain three lattice points over $(1,0)$. 

Let us also consider the case of $|\interior_{\Z}(F)| = 1$. Here, we can fit $F'$ inside the standard square $[0,2]^2$ (inside the $x$-$z$-plane) after a suitable shear fixing the $x = 1$ line, so $F'$ can be taken to be one of the following polygons:

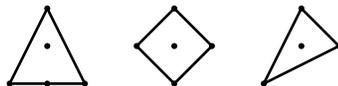
\begin{center}
		\begin{tikzpicture}[scale=0.5,baseline=-5pt] 
		\draw[line width=1pt] (1, 2) -- (2, 0);\draw[line width=1pt] (0, 0) -- (2, 0);\draw[line width=1pt] (0, 0) -- (1, 2);\node at (1, 2/3) {}; 
		\fill (0,0) circle (0.08);\fill (1,0) circle (0.08);\fill (1,1) circle (0.08);\fill (1,2) circle (0.08);\fill (2,0) circle (0.08);
		\end{tikzpicture}
		\hspace{10pt}
		\begin{tikzpicture}[scale=0.5,baseline=-5pt] 
		\draw[line width=1pt] (1, 0) -- (2, 1);\draw[line width=1pt] (1, 2) -- (2, 1);\draw[line width=1pt] (0, 1) -- (1, 2);\draw[line width=1pt] (0, 1) -- (1, 0);\node at (1, 1) {}; 
		\fill (0,1) circle (0.08);\fill (1,0) circle (0.08);\fill (1,1) circle (0.08);\fill (1,2) circle (0.08);\fill (2,1) circle (0.08);
		\end{tikzpicture}
		\hspace{10pt}
		\begin{tikzpicture}[scale=0.5,baseline=-5pt] 
		\draw[line width=1pt] (1, 2) -- (2, 1);\draw[line width=1pt] (0, 0) -- (2, 1);\draw[line width=1pt] (0, 0) -- (1, 2);\node at (1, 1) {}; 
		\fill (0,0) circle (0.08);\fill (1,1) circle (0.08);\fill (1,2) circle (0.08);\fill (2,1) circle (0.08);
		\end{tikzpicture}
\end{center}

Again, only the last one does not contain three lattice points over $(1,0)$. 

Lastly, for each of these two remaining polygons $F'$ we may choose a lattice subpolytope $P'$ of $P$ that is a pyramid of height $2$ over $F'$. We observe that for given $F'$ there are at most four non-isomorphic possibilities for $P'$ to consider as the first two coordinates of an apex in $\R^2 \times \{2\}$ over a base polytope in $\R^2 \times \{0\}$ may be chosen by a unimodular shearing to be in $\{(0,0),(1,0),(0,1),(1,1)\}$. Now, an explicit computation in \texttt{SageMath} shows that for all these at most eight cases we have $|\interior_{\Z}(2P')| > |\interior_{\Z}(F')| = |\interior_{\Z}(F)|$, concluding the proof.
\end{proof}

We can now answer the original question in \cite{GKZ94} in dimension $3$. 

\begin{cor}
A three-dimensional lattice simplex is thin if and only if it is a lattice pyramid.\label{cor:tetra}
\end{cor}

This follows directly from Theorem~\ref{thm:3d}. The reader is cautioned not to jump to the conclusion that the same result may be true in higher dimensions. In dimension $4$, $[-1,1] \circ_\Z 2 \Delta_2$ is an example of a (trivially) thin simplex that is not a lattice pyramid.

\section{Thin Gorenstein polytopes and Gorenstein joins}
\label{sec:join}

\subsection{Gorenstein polytopes and their duals}

\begin{defn}
A lattice polytope $P$ is called {\em Gorenstein} if $h^*_P$ is palindromic.
\end{defn}

Let us recall that {\em reflexive polytopes} are precisely the Gorenstein polytopes of codegree one. For more background on reflexive and Gorenstein polytopes, its relevance in toric geometry and mirror symmetry, as well as alternative characterizations we refer to \cite{BB96b, batyrev2008combinatorial, NS13}. Here, let us summarize definitions and properties of the dual Gorenstein polytope. We remark that the codegree of a Gorenstein polytope is often called its index.

\begin{defn}
    Let $P \subset \R^d$ be a $d$-dimensional Gorenstein polytope. In this case, the dilate $\codeg(P) \cdot P$ is a reflexive polytope (up to lattice translation), and we denote its unique interior lattice point by $w$. Then 
\[P^\times := \{y \in (\R^{d+1})^* \;:\; \pro{y}{w} = 1 \text{ and } \pro{y}{x} \ge 0 \;\forall\, x \in P \times \{1\}\}\]
is called the {\em dual Gorenstein polytope} of $P$. 
\label{def:dual-def}
\end{defn}

\begin{prop}
Let $P \subset \R^d$ be a $d$-dimensional Gorenstein polytope. Then $P^\times$ is a Gorenstein polytope of the same dimension and degree as $P$, and it is combinatorially dual to $P$. 
\label{prop:dual-def}
\end{prop}

Note that, if a Gorenstein polytope is lower-dimensional, we consider, as usual, its ambient lattice in order to get its dual Gorenstein polytope.

\begin{defn}
If $F$ is a face of $P$, we denote by $F^*$ the {\em dual face}, i.e., the corresponding face of $P^\times$.
\end{defn}

Attention: it is important to distinguish the dual face $F^\ast$ from $F^\times$, the latter being defined only if $F$ is itself a Gorenstein polytope which is not true in general. Even if this is the case, the two polytopes might have completely different dimensions (since the one definition is relative to $P$ while the other one is intrinsic). 

Local $h^*$-polynomials of Gorenstein polytopes (often called $\tilde{S}$-polynomials) allow to give an elegant formula for computing stringy $E$-polynomials of Calabi-Yau complete intersections in toric Gorenstein Fano varieties (we refer to \cite{BorisovMavlyutov, batyrev2008combinatorial}). In this context, several questions about stringy $E$-polynomials are still open, see \cite{batyrev2008combinatorial,NS13}. Here, we make some progress in this direction by addressing the question when the local $h^*$-polynomial of a Gorenstein polytope vanishes. As one consequence of our main result, Theorem~\ref{thm:main}, we will see that not only the degree of the $h^*$-polynomials of Gorenstein polytopes and their duals are the same but also of their $\ell^*$-polynomials (Corollary~\ref{cor:dualdeg}).

\subsection{Joins, Cayley polytopes, and Cayley joins}

In the sequel let us discuss some important notions of decomposing lattice polytopes that turn up naturally when studying Gorenstein polytopes (we refer to \cite{batyrev2008combinatorial, NS13}). Let us first introduce a formal notation for a lattice polytope being a join (as already defined in Subsection~\ref{subsec:ex}).

\begin{defn}
	Let $P \subseteq \R^d$ be a polytope and $F, G$ non-empty subsets of $P$. Then $P$ is the {\em join of $F$ and $G$}, written $P = F \circ G$, if $P = \conv(F,G)$ and $\dim(P) = \dim(F) + \dim(G) + 1$. 
\end{defn}

Equivalently, $P$ is affinely-isomorphic to the free join $F \circ_\Z G$. In particular, $F$ and $G$ are automatically faces of $P$.

\begin{rmk}{\rm Note that the join property is associative. Namely, given faces $F,G,H$ of $P$, then $P=F \circ (G \circ H)$, respectively, $P=(F \circ G) \circ H$, are both equivalent to $P=\conv(F,G,H)$ and $\dim(P) = \dim(F) + \dim(G) + \dim(H) + 2$.\label{remark:assoc}
}
\end{rmk}

Let us also give the formal notation for a lattice polytope being a Cayley polytope. We recall that the notion of Cayley polytopes and Cayley sums was already shortly mentioned and defined in Subsection~\ref{subsec:two}. Here, we will solely focus on the case of two factors. Note that if a Cayley polytope has more than two factors, it is still a Cayley polytope with two factors.

\begin{defn}
	Let $P \subseteq \R^d$ be a lattice polytope and $F, G$ non-empty subsets of $P$. Then $P$ is the {\em Cayley polytope} of (factors) $F$ and $G$, written $P = F * G$, if $P = \conv(F,G)$ and there exists an affine-linear map $\R^d \to \R$ mapping $\Z^d \to \Z$, such that $F \mapsto 0$ and $G \mapsto 1$. In other words, $P$ is a Cayley polytope if and only if there is a lattice projection mapping $P$ onto $\Delta_1$.
	\label{def:cayley}
\end{defn}

If $P=F*G$, then $F$ and $G$ are necessarily faces of $P$. Cayley polytopes can also be characterized as lattice polytopes with lattice width one. Cayley sums are explicit descriptions of Cayley polytopes.

\begin{rmk}
Given lattice polytopes $F$ and $G$ in $\R^d$, the convex hull of $F \times \{0\}$ and $G \times \{1\}$ is called the {\em Cayley sum} of $F$ and $G$. Its dimension is one larger than the dimension of the Minkowski sum of $F$ and $G$. If $P=F*G$, then $P$ is isomorphic to the Cayley sum of $F$ and $G$.\label{rem:cayley}
\end{rmk}

Cayley sums are important in the construction of high-dimensional Gorenstein polytopes, see e.g. \cite[Theorem~2.6]{batyrev2008combinatorial}. Note that the degree of a Cayley polytope is at most the dimension of the Minkowski sum of its factors, see Proposition~\cite[Proposition~1.12]{batyrev2007multiples}. 

\begin{defn}
	Let $P \subseteq \R^d$ be a full-dimensional lattice polytope and $F,G \subseteq P$ faces. Then $P$ is the {\em Cayley join} of $F$ and $G$, written $P = F \circ_{\Cay} G$, if $P=F \circ G$ and $P=F * G$.\label{def:cayley-join}
\end{defn}

Clearly, the notion of a Cayley join is more restrictive than that of a Cayley polytope (e.g., $[0,1]^2$ is a Cayley polytope of two edges but not a Cayley join). The reader should be aware that Cayley polytopes and Cayley joins are in general not associative in the sense of Remark~\ref{remark:assoc}, see Example~\ref{ex:tetra} below.

\smallskip
Let us recall some properties of a Gorenstein polytope that is a join or Cayley join, see \cite[Lemma~4.8, Proposition~4.9]{NS13}.


\begin{prop}
\label{prop:NS-prop}
	Let $P$ be a Gorenstein polytope. 
	\begin{itemize}
	    \item If $P=F \circ G$, then $P^\times = F^* \circ G^*$ with $F$ and $G^*$ (respectively, $G$ and $F^*$) being combinatorially dual to each other.
\item	If $P=F \circ_{\Cay} G$, then $F^\ast$ is a Gorenstein polytope with dual Gorenstein polytope $(F^\ast)^\times$, which can be identified with the lattice polytope $G$ with respect to a refined lattice.
	\end{itemize}
	In the last statement, the lattice does not have to be refined if the Cayley join is even a free join.
\end{prop}

\subsection{Gorenstein joins}

In \cite{NS13} the following notion was defined. 

\begin{defn}\label{def:Gorenstein-join}
Let $F$ and $G$ be faces of a Gorenstein polytope $P$. We say $P$ is a {\em Gorenstein join} of $F$ and $G$, denoted by $P = F \circ_{\Gor} G$, if $P = F \circ_{\Cay} G$ and $P^\times = F^* \circ_{\Cay} G^*$. We call $F$ and $G$ the {\em factors} of the Gorenstein join. 
\end{defn}

We remark that Gorenstein joins do not have to be free joins, see \cite[Example~4.14]{NS13}. The following result, a strengthening of Stanleys monotonicity theorem in the case of faces of Gorenstein polytopes, motivated the previous definition of a Gorenstein join and gives a direct enumerative criterion for its existence (\cite[Theorems~3.6+4.12]{NS13}).

\begin{thm}
    Let $P$ be a Gorenstein polytope and $F$ a non-empty proper face of $P$. Then $\codeg(P) \le \codeg(F) + \codeg(F^*)$ (equivalently, $\deg(F) + \deg(F^*) \le \deg(P)$), with equality if and only if $P$ is a Gorenstein join with factor $F$. In this case, $F$ is a Gorenstein polytope.
    \label{thm:NS}
\end{thm}

Gorenstein polytopes that are not Gorenstein joins have been previously also called {\em irreducible} in \cite{NS13}. As we see from the following result it is not necessary to compute the dual Gorenstein polytope to check whether a Cayley join is a Gorenstein join.

\begin{lem}\label{lemma:GorensteinCodegree}
	Let $P = F \circ_{\Cay} G$ be a Gorenstein polytope which is the Cayley join of two faces $F, G \leq P$. Then $P = F \circ_{\Gor} G$ if and only if $\codeg(P) = \codeg(F) + \codeg(G)$ (or equivalently, $\deg(P) = \deg(F)+\deg(G)$).
\end{lem}

\begin{proof}
	By Theorem~\ref{thm:NS}, $P = F \ast_{\Gor} G$ if and only if $\codeg(F) + \codeg(F^\ast) = \codeg(P) =: r$, and in this case by \cite[Theorem~4.12]{NS13} $\codeg(G) = \codeg(P) - \codeg(F) = \codeg(F^\ast)$. Conversely, assume $\codeg(G) = \codeg(P) - \codeg(F)$. By Theorem~\ref{thm:NS}, the inequality $\codeg(F) + \codeg(F^\ast) \geq r$ always holds, so that by Theorem~\ref{thm:NS} again we only need to prove $\codeg(F^\ast) \leq \codeg(G)$. As $P$ is the Cayley join of $F$ and $G$, Proposition~\ref{prop:NS-prop} yields that $\codeg(F^\ast) = \codeg((F^\ast)^\times) \leq \codeg(G)$ as the codegree can only decrease under refinements of the lattice.
\end{proof}

\begin{rmk}As we will need it for the upcoming proofs, let us recall how to characterize Gorenstein polytope via cones. For more details, we refer to \cite{batyrev2008combinatorial}. The {\em cone over $P$} is denoted by $C_P \subseteq \R^{d+1}$ spanned by $P \times \{1\} \subset \R^{d+1}$. Any polyhedral cone in $\R^{d+1}$ that is unimodularly equivalent to some $C_P$ is called a {\em Gorenstein cone}. Now, $P$ is a Gorenstein polytope if and only if the dual cone $C_P^\vee = \{y \in (\R^{d+1})^*: \langle y, x \rangle \ge 0 \;\forall\, x \in C_P\}$ is a Gorenstein cone. In this case, $C_P^\vee$ is unimodularly equivalent to the cone over $P^\times$.
\end{rmk}

The following proposition contains a positive result regarding associativity of Gorenstein joins. In general, however, we do not expect associativity to hold.

\begin{prop}
	Let $P \subseteq \R^d$ be a $d$-dimensional Gorenstein polytope with faces $F, G, H \leq P$ such that $P = (F \ast_{\Gor} G) \ast_{\Gor} H$. If $F$ is a vertex or $H$ is a vertex, then $P = F \ast_{\Gor} (G \ast_{\Gor} H)$. In particular, Gorenstein joins are associative for $\dim(P) \leq 3$.
\end{prop}

\begin{proof}
	 By Lemma~\ref{lemma:GorensteinCodegree}, the faces $\conv(F,G)$ and $H$ of $P$ are themselves Gorenstein polytopes with $r := \codeg(P) = \codeg(\conv(F,G)) + \codeg(H)$. Applying the same result to the Gorenstein polytope $\conv(F,G) = F \ast_{\Gor} G$, we obtain that $F$ and $G$ are Gorenstein polytopes with $\codeg(\conv(F,G)) = \codeg(F) + \codeg(G)$. Hence, $r = \codeg(F) + \codeg(G) + \codeg(H)$. By Lemma~\ref{lemma:GorensteinCodegree}, it hence suffices to show $P = F \ast_{\Cay} (G \ast_{\Cay} H)$. That the join of $G$ and $H$ is a Cayley join is immediate from the assumption, so it is enough to show that the join of $F$ and $\conv(G,H)$ is a Cayley join.\medskip
	
	Let first $F$ be a vertex. Hence, $P$ is a pyramid with vertex $F$ and base $\conv(G,H)$. Consider the cone $C_P \subset \R^{d+1}$. As $P$ is Gorenstein of codegree $r$, there is a unique interior lattice point $p \in C_P$ on height $p_{d+1} = r$. Similarly, let $f$, $g$ and $h$ be the unique interior lattice points of $C_F, C_G, C_H \subseteq C_P$ on heights $\codeg(F), \codeg(G), \codeg(H)$, respectively. Then necessarily $f+g \in C_{\conv(F,G)} \subseteq C_P$ is the unique interior lattice point on height $\codeg(\conv(F,G))$ of the Gorenstein polytope $\conv(F,G)$. Hence, $p = f + g + h$. Let $u \in (\Z^{d+1})^*$ be the primitive inner facet normal of the hyperplane containing $\conv(G,H)$. As $u$ is a vertex of $P^\times$, we have $\langle u, p \rangle = 1$ (cf. Definition~\ref{def:dual-def}). Therefore,
	\begin{equation*}
	1 = \langle u, p \rangle = \langle u, f \rangle + \langle u, g+h \rangle = \langle u, f \rangle.
	\end{equation*}
	But this means that $F$ and $\conv(G,H)$ have lattice distance equal to $1$, i.e., the combinatorial join $F \circ \conv(G,H)$ is a Cayley join.\medskip

	Let now $H$ be a vertex, so $P$ is a lattice pyramid with vertex $H$ and base $\conv(F,G)$. As $F \circ_{\Cay} G$, we may assume that ${\rm lin}(F,G)=\R^{d-1} \times\{0\}$, $H = \{e_d\}$, and there exists some $u \in (\Z^{d-1} \times \{0\})^*$ such that $\langle u, F \rangle = 0$ and $\langle u, G \rangle = 1$. Now, $\langle u + e^*_d, F \rangle = 0$ and 
	$\langle u + e^*_d, \conv(G,H) \rangle = 1$. In particular, the join of $F$ and $\conv(G,H)$ is a Cayley join.
	
	\smallskip
	Finally, if $d = \dim(P) \leq 3$ and $P$ is the join of $F$, $G$ and $H$, then necessarily at least one of $F$ and $H$ is a vertex for dimension reasons, concluding the proof.
\end{proof}

We observe that a Gorenstein polytope $P$ is a lattice pyramid over a face $F$ with apex a vertex $v$ of $P$ if and only if $P$ is a Gorenstein join of $F$ and $v$. Hence, the previous result has the following consequence.

\begin{cor}
If $P$ is a Gorenstein join of two faces with one face a lattice pyramid, then $P$ is also a lattice pyramid.
\label{cor:pyr}
\end{cor}

This result is already contained in the master's thesis \cite{masterthesis}. Let us give an example that shows that it fails for Cayley joins.

\begin{ex}
Consider $F := \conv(e_1,e_2) \times \{0\}$ and $G := \conv(0,-e_1-e_2) \times \{1\}$ in $\R^3$. Then its convex hull $P$ is a tetrahedron that is a Gorenstein polytope of lattice volume~$2$ (with $h^*_P(t)=1+t^2$ and $\ell^*_P(t)=t^2$). It is a Cayley join $P = F \circ_{\Cay} G$ but not a Gorenstein join as $\deg(P)=2\not=0=\deg(F)+\deg(G)$. Note that $F$ and $G$ are lattice pyramids, but $P$ is not. In particular, this example shows that the Cayley join property is not associative, and moreover, a Cayley join does not have to be thin if a factor of the Cayley join is thin.\label{ex:tetra}
\end{ex}

\subsection{Local $h^*$-polynomials of joins}

Recall that $h^\ast$- and $\ell^\ast$-polynomials are multiplicative with respect to free joins (Proposition~\ref{prop:mult}). For general joins, one still gets inequalities. 

\begin{lem}\label{lemma:inequalities}
    Let $P$ be a lattice polytope which is the join of two faces $F$ and $G$. Then
    \begin{equation*}
        \ell^\ast_F(t) \cdot \ell^\ast_G(t) \leq \ell^\ast_P(t) \ \text{ and } \ h^\ast_F(t) \cdot h^\ast_G(t) \leq h^\ast_P(t).
    \end{equation*}
    If moreover $P$ is a Gorenstein polytope which is the Gorenstein join of $F$ and $G$, then also
    \begin{equation*}
        \ell^\ast_{P^\times}(t) \leq \ell^\ast_{F^\times}(t) \cdot \ell^\ast_{G^\times}(t) \ \text{ and } \ h^\ast_{P^\times}(t) \leq h^\ast_{F^\times}(t) \cdot h^\ast_{G^\times}(t).
    \end{equation*}
\end{lem}

\begin{proof}
We will use the notation of \cite{NS13}.Let $\overline{M} = \Z^{d+1}$, and $M(F)$ denote the sublattice of $\overline{M}$ spanned by the lattice points in the linear hull of $F \times \{1\}$. Relative to the sublattice $M(F) \oplus_\Z M(G)$ of $\overline{M}$, $P$ becomes the free join of $F$ and $G$. Recall that by Corollary~\ref{cor:lattice} both the $h^\ast$-polynomial and the $\ell^\ast$-polynomial are (weakly) monotonically increasing under refinements of the lattice. It hence follows from Proposition~\ref{prop:mult} that $\ell^\ast_F(t) \cdot \ell^\ast_G(t) \leq \ell^\ast_P(t)$ and $h^\ast_F(t) \cdot h^\ast_G(t) \leq h^\ast_{P}(t)$.\medskip

For the second claim, assume that $P \subseteq \R^d$ is a full-dimensional Gorenstein polytope of codegree $r$ with respect to the lattice $M = \Z^d \subseteq \R^d$. By assumption, $P$ is the Gorenstein join of two faces $F$ and $G$. By Theorem~\ref{thm:NS} and Lemma~\ref{lemma:GorensteinCodegree}, $F$ and $G$ are Gorenstein polytopes and $\codeg(F) + \codeg(G) = r$. For $\overline{M} = \Z^{d+1}$ we define the dual lattice $\overline{N} \coloneqq \mathrm{Hom}_\Z(\overline{M},\Z) \subseteq (\R^{d+1})^*$. By definition, as $P$ is a Gorenstein polytope, $C_P^\vee$ is a Gorenstein cone with respect to $\overline{N}$. Let $n = e_{d+1}^\ast \in \overline{N}$ be the unique interior lattice point of $C_P^\vee$ with $P \times \{1\} = C_P \cap \{x \in \R^{d+1}: \langle n, x \rangle = 1\}$. In the same way, we denote by $m \in \overline{M}$ the unique interior lattice point of $C_P$ such that $P^\times = C_P^\vee \cap \{y \in (\R^{d+1})^*: \langle y, m \rangle = 1\}$. Recall that $\langle n, m \rangle = r$ and hence $m = (p,r) \in M \oplus \mathbb{Z} = \overline{M}$ with $p \in M$ the unique interior lattice point of the $r$-th dilate $rP$ of $P$. 

Now, let us consider the sublattice $M(F) \oplus M(G) \subseteq \overline{M}$. With respect to this coarser lattice, the polytope $P$ is the \emph{free} join of $F$ and $G$, and this is clearly a Gorenstein polytope of codegree $\codeg(F) + \codeg(G) = r$. Hence, the $r$-th dilate $rP$ contains a unique interior lattice point in the original as well as in the coarser lattice. These two points must therefore agree, so the unique interior lattice point $m = (p,r) \in (rP) \times \{r\} \subseteq C_P$ with respect to the original lattice actually lies in $M(F) \oplus M(G)$. Let now $\Tilde{N} \subset (\R^{d+1})^*$ be the dual lattice of $M(F) \oplus M(G)$. Hence, $P^\times$ is with respect to the finer lattice $\Tilde{N}$ the dual Gorenstein polytope of the Gorenstein polytope $P$ considered with respect to the coarser lattice $M(F) \oplus M(G)$. By Proposition~\ref{prop:NS-prop}, the Gorenstein dual of the free join of the Gorenstein polytopes $F$ and $G$, is the free join of the Gorenstein duals $F^\times$ and $G^\times$. Again, monotonicity and multiplicativity proves the second claim: $\ell^\ast_{P^\times}(t) \leq \ell^\ast_{F^\times}(t) \cdot \ell^\ast_{G^\times}(t)$ and $h^\ast_{P^\times}(t) \leq h^\ast_{F^\times}(t) \cdot h^\ast_{G^\times}(t)$, where $P^\times$ is considered with respect to the original lattice $\overline{N}$ again.
\end{proof}

\begin{rmk}
    It follows in the situation of the second part of Lemma~\ref{lemma:inequalities} from Proposition~\ref{prop:NS-prop} and Remark~\ref{rmk:refine} that $\ell^\ast_{F}(t) \leq \ell^\ast_{(G^\ast)^\times}(t)$ and $\ell^\ast_G(t) \leq \ell^\ast_{(F^\ast)^\times}(t)$, because $(G^\ast)^\times$ is just the polytope $F$ with a possibly finer lattice, and analogously for $(F^\ast)^\times$ and $G$. The same holds for the $h^\ast$-polynomial.
\end{rmk}

\section{Characterization of thin Gorenstein polytopes}

\subsection{The main result}

\label{sec:gorst}

The following notion will occur naturally in the proof of Theorem~\ref{thm:main}.

\begin{defn}
A lattice polytope $P$ is called {\em $g$-thin} if $\deg(g_P) = \deg(P)$. 
\end{defn}

For instance, any unimodular simplex is $g$-thin. Note that by Corollary~\ref{cor:schepers}, we always have $\deg(g_P) \le \deg(P)$. As by construction (Definition~\ref{def:fgh}) $\deg(g_P) \le \dim(P)/2$, we observe that
\[g\text{-thin} \quad\lora\quad \text{trivially thin}\]

\begin{ex}
Let $P$ denote the lattice pyramid over $[-1,1]$. Then $P$ has dimension $2$, degree $1$ and $\deg(g_P)=0$ as it is a simplex. This is an example of a trivially thin Gorenstein polytope that is not $g$-thin. For another example, consider $2 \Delta_2$ which is a trivially thin (non-Gorenstein) simplex that is not $g$-thin. This example shows that a spanning thin polytope that is not a free join does not have to be $g$-thin (i.e., in Question~\ref{question} 'trivially thin' cannot be strengthened by 'g-thin').
\end{ex}

Here is our main result. 

\begin{thm}\label{thm:main}
    Let $P$ be a Gorenstein polytope. Then the following are equivalent:
    \begin{enumerate}[label=(\roman*)]
        \item $P$ is thin,\label{item:thin}
        \item $P$ is trivially thin or $P = F \ast_{\Gor} G$ with at least one factor trivially thin,\label{item:trivially_thin}
        \item $P$ is $g$-thin or $P = F \ast_{\Gor} G$ with $\deg(\ell^*_F) = \deg(F)$ and $G$ $g$-thin.\label{item:g-thin}
    \end{enumerate}
    Moreover, if $P$ is not thin, then $\deg(\ell_P^\ast) = \deg(P)$.
\end{thm}

Let us remark that if $P$ is not thin, the last statement implies that $\ell^*_P(t)$ and $h^*_P(t)$ have the same leading coefficient $1$, as $h^*_P(t)$ is palindromic with constant coefficient $1$.

We have to leave it as an open question whether it is possible to strengthen in the previous result `Gorenstein join' to `free join'. Let us note the following situation in which being a Gorenstein join (or even just a spanning join of faces) automatically implies being a free join.

\begin{cor}
Let $P$ be a spanning Gorenstein polytope. Then $P$ is thin if and only if it is trivially thin or a free join with a trivially thin factor (necessarily also a spanning Gorenstein polytope).\label{cor:spanning}
\end{cor}

The proof of Theorem~\ref{thm:main} relies critically on the decomposition of the $h^*$-polynomial into $\ell^*$-polynomials and $g$-polynomials (Corollary~\ref{cor:schepers}), valid also for general lattice polytopes.

\begin{lem}
Let $P$ be a lattice polytope with $\deg(\ell^*_P) < \deg(P)$. Then $P$ is $g$-thin or there exists a non-empty, proper face $F$ of $P$ with 
$\deg(P)=\deg(\ell^*_F)+\deg(g_{[F,P)})$.
\label{lem:scheperscons}
\end{lem}

Here, we recall $\deg(\ell^*_{\emptyset})=\deg(1)=0$ and $\deg(g^*_{\emptyset})=\deg(1)=0$.

\begin{proof}
By the nonnegativity of $\ell^*$- and $g$-polynomials, Corollary~\ref{cor:schepers} implies that there exists a face $F$ of $P$ with $\deg(\ell^*_F)+\deg(g_{[F,P)}) = \deg(h^*_P)$. By our assumption, $F\not=P$. If $F=\emptyset$, then $\deg(h^*_P) = \deg(g_{[\emptyset,P)})$, so $P$ is $g$-thin.
\end{proof}

Let us note the following observation:

\begin{lem}
Let $P = F \ast_{\Gor} G$. Then the Gorenstein polytopes $F$, $G^*$, $F^\times$, $(G^*)^\times$ have the same degree, dimension, and degree of their $g$-polynomials. In particular, if any of these Gorenstein polytopes are trivially thin (respectively, $g$-thin), then all of them are.
\label{lem:deg}
\end{lem}

\begin{proof}
The Gorenstein property follows from Proposition~\ref{prop:NS-prop}. It is well-known, cf. \cite{batyrev2008combinatorial}, that duality of Gorenstein polytopes keeps dimension and degree invariant. It follows from Theorem~\ref{thm:Kalai} that this is also true for the degree of the $g$-polynomial. By Proposition~\ref{prop:NS-prop}, $F$ and $G^*$ are combinatorially dual to each other, hence, have the same dimension and by Theorem~\ref{thm:Kalai} the same degree of the $g$-polynomial. Finally, by Lemma~\ref{lemma:GorensteinCodegree} and Theorem~\ref{thm:NS},
\[\deg(P) - \deg(F) = \deg(G) = \deg(P) - \deg(G^*),\]
hence, $\deg(F)=\deg(G^*)$. 
\end{proof}

\begin{proof}[Proof of Theorem~\ref{thm:main}]
The implication \ref{item:g-thin} $\Rightarrow$ \ref{item:trivially_thin} is immediate.

\ref{item:trivially_thin} $\Rightarrow$ \ref{item:thin}: 
Let $P = F \ast_{\Gor} G$ with $F$ trivially thin. By Lemma~\ref{lem:deg}, it follows that $(G^\ast)^\times$ is trivially thin as well. Now, applying Lemma~\ref{lemma:inequalities} to the factorization $P^\times = F^\ast \ast_{\Gor} G^\ast$ yields $\ell^\ast_P(t) \leq \ell^\ast_{(F^\ast)^\times}(t) \cdot \ell^\ast_{(G^\ast)^\times}(t) = 0$, so $P$ is thin.

\ref{item:thin} $\Rightarrow$ \ref{item:g-thin}:
Let $P$ be a Gorenstein polytope. We assume only that $\deg(\ell_P^\ast) < \deg(P)$ and will deduce \ref{item:g-thin}, so that $P$ is in particular thin by the implications we already proved (and thus, if $P$ is not thin, then $\deg(\ell^\ast_P) = \deg(P)$). Let us assume that $P$ is not $g$-thin. Now, by Lemma~\ref{lem:scheperscons} there exists a non-empty, proper face $F$ of $P$ with $\deg(P)=\deg(\ell^*_F)+\deg(g_{[F,P)})$. Theorem~\ref{thm:Kalai} 
shows that $\deg(g_{[F,P)})=\deg(g_{(F, P]^*}) = \deg(g_{F^*})$. Thus, Corollary~\ref{cor:schepers} and Theorem~\ref{thm:NS} imply that
\[\deg(F^*) \ge \deg(g_{F^*}) = \deg(P) - \deg(\ell^*_F) \ge 
\deg(P) - \deg(F) \ge \deg(F^*).\]
Therefore, $F^*$ is $g$-thin, $\deg(\ell^*_F)=\deg(F)$, and $\deg(F)+\deg(F^*)=\deg(P)$, which implies \ref{item:g-thin} by Theorem~\ref{thm:NS} and Proposition~\ref{prop:NS-prop} (with the roles of $F$ and $G$ exchanged).
\end{proof}

\begin{cor}
Every thin Gorenstein polytope (of dimension $>0$) has lattice width $1$.\label{cor:flat-gorst}
\end{cor}

\begin{proof}
As Gorenstein joins are Cayley joins, it remains by Theorem~\ref{thm:main} to show that a trivially thin Gorenstein polytope of dimension $>0$ is a Cayley polytope. This is precisely the statement of Theorem~3.1 in \cite{HNP09}.
\end{proof}

\begin{ex}
Let us illustrate Theorem~\ref{thm:main} by showing that all thin Gorenstein polytopes $P$ of dimension $d=3$ are lattice pyramids over Gorenstein polygons (without using Theorem~\ref{thm:3d} directly). Let us assume otherwise. If $P$ is trivially thin, then $\deg(P) \le 1$, so by Theorem~\ref{thm:BN} $P$ is a Lawrence prism. Palindromicity implies $h^*_P(t)=1 + t$, so lattice volume $2$, which is a contradiction because any three-dimensional Lawrence prism has at least lattice volume $3$. Hence, by Theorem~\ref{thm:main} $P$ must be a lattice pyramid or a Gorenstein join of two Gorenstein intervals one of them being thin. As a thin interval is a unimodular simplex, thus a lattice pyramid, also $P$ is a lattice pyramid by Corollary~\ref{cor:pyr}. 
\end{ex}

\subsection{Borisov's proof of the degree of $\ell^*$-polynomials of non-thin Gorenstein polytopes}

Theorem~\ref{thm:main} answers affirmatively Question~6.3(b) in \cite{NS13} asking whether for Gorenstein polytopes having a non-vanishing $\ell^*$-polynomial forces its degree to be maximal (i.e., equal to the degree of the $h^*$-polynomial). Lev Borisov has provided us with an alternative algebraic proof of this fact that we reproduce here. It uses the description of the local $h^*$-polynomial of a lattice polytope as a Hilbert series of a graded ideal given in \cite{BorisovMavlyutov}.

\begin{prop}Let $P \subseteq \R^d$ be a Gorenstein polytope of codegree $r$. Then either $P$ is thin or $\ell^\ast_P(t)$ starts with $t^r$. In this case, $\ell^\ast_{P}(t)$ has degree $\deg(P)$ and leading coefficient~$1$.\label{prop:borisov}
\end{prop}

\begin{proof}
Let $K \subseteq \Z^{d+1}$ be the lattice points in the Gorenstein cone over $P \times \{1\}$. Denote by $\C[K]$ the associated affine semi-group algebra with $\mathbb{N}_0$-grading given by the exponent of $x_{d+1}$, viewing $\C[K] \subseteq \C[x_1^{\pm 1}, \ldots, x_d^{\pm 1}, x_{d+1}]$.
As in \cite[Section~4]{BorisovMavlyutov}, we let $f \in \C[K]_1$ be non-degenerate and $I \subseteq \C[K]$ the homogeneous ideal generated by the so called logarithmic derivatives of $f$.
Let moreover $J \subseteq \C[K]$ be the homogeneous ideal generated by all lattice points in the relative interior $K^\circ$ of $K$.
Then Borisov and Mavlyutov define $R_1(f,K)$ to be the image of $J$ in the quotient ring $\C[K]/I$, i.e., $R_1(f,K)$ is the homogeneous ideal $(I+J)/I$ of $\C[K]/I$.\\
Now, by \cite[Proposition~5.5]{BorisovMavlyutov}, $\ell^\ast_P(t)$ is the Hilbert series of $R_1(f,K)$.
Moreover, as $P$ is Gorenstein of codegree $r$, we have $K^\circ = (p,r) + K$, where $p \in (rP) \cap \Z^d$ is the unique interior lattice point of $rP$. Therefore, $R_1(f,K)$ is just the image of the principal ideal $(x^p x_{d+1}^r)$ in the quotient $\C[K]/I$. Hence, $R_1(f,K)$ is $0$ if and only if $x^p x_{d+1}^r \in I$, and otherwise the lowest degree of its non-zero homogeneous components is $r$ with $R_1(f,K)_r = \langle x^p x_{d+1}^r \rangle$ of vector space dimension $1$. This proves the first claim, and the second follows from reciprocity, Theorem~\ref{thm:lstar}(2).
\end{proof}

In dimensions $\leq 4$ it is a consequence of the reciprocity of $\ell^\ast_P(t)$ that for any lattice polytope $P$ either $P$ is thin or $\deg(\ell^\ast_P) = \deg(P)$. In higher dimensions, this property fails for non-Gorenstein lattice polytopes.

\begin{ex}
	Consider the full-dimensional lattice simplex $P \subseteq \R^5$ given as the convex hull $P = \conv(0, e_1, e_2, e_3, (0,1,1,2,0), (5,3,3,2,6))$. Then $\ell_P^\ast(t) = 4t^3$ while $h^\ast_P(t) = t^4 + 5t^3 + 4t^2 + t + 1$. In particular, $P$ is not thin but $\deg(\ell^\ast_P) < \deg(h^\ast_P) = \deg(P)$. This is the only such example among lattice simplices of dimension $5$ with lattice volume $\leq 15$. It was found using the database \cite{gabriele}. The computations were performed in \texttt{SageMath} with backend \texttt{Normaliz}.
\end{ex}

\begin{ex}
    Consider the full-dimensional lattice simplex $P \subseteq \R^5$ given as the convex hull $P = \conv(0, e_1, e_2, (1, 1, 2, 0, 0), (3, 5, 6, 8, 0), (1, 1, 0, 0, 2))$. Then $\ell_P^\ast(t) = t^3$ while $h^\ast_P(t) = 7t^3 + 19t^2 + 5t + 1$. So $\deg(\ell_P^\ast) = \deg(h_P^\ast)$ but the leading coefficient of $\ell_P^\ast$ is strictly smaller.
\end{ex}

\subsection{Thinness is invariant under duality}

It was noted in Lemma~\ref{lem:deg} that being trivially thin as well as being $g$-thin is invariant under duality of Gorenstein polytopes. Let us explain how this allows us to deduce that also thinness has this beautiful duality property:

\begin{cor}
Let $P$ be a Gorenstein polytope. Then $P$ is thin if and only if $P^\times$ is thin.\label{cor:dual}
\end{cor}

\begin{proof}

By Theorem~\ref{thm:main}(ii) we may assume that $P$ is a thin Gorenstein polytope such that $P$ is a Gorenstein join of faces $F$ and $G$ with $F$ being trivally thin. Hence, by Lemma~\ref{lem:deg} we also have $P^\times = F^* \ast_{\Gor} G^*$ with $G^*$ being trivially thin. Again, by Theorem~\ref{thm:main} this implies that $P^\times$ is thin.
\end{proof}

Having such a direct proof answers a question of Lev Borisov, who communicated to us that this statement might also be proven using vertex algebra techniques. 

In particular, as Theorem~\ref{thm:main} implies that there are only two choices for the degree of the $\ell^*$-polynomial of a Gorenstein polytope we see that its degree is also invariant under duality (as it holds for the degrees of the $h^*$-polynomial and the $g$-polynomial).

\begin{cor}
Let $P$ be a Gorenstein polytope. Then $\deg(\ell^*_P) = \deg(\ell^*_{P^\times})$.
\label{cor:dualdeg}
\end{cor}

\begin{ex}
The reader should be aware that the local $h^*$-polynomials of a Gorenstein polytope $P$ and its dual $P^\times$ may differ. For instance, for $P=[-1,1]^3$ we have $\ell^*_P(t) = t + 17 t^2 + t^3$ and $\ell^*_{P^\times}(t)=t + 3 t^2 + t^3$.
\end{ex}

\subsection{Thin Gorenstein simplices}

For the special case of Gorenstein simplices, we can answer the original question in \cite{GKZ94} about classifying thin simplices.

\begin{cor}
Let $P$ be a Gorenstein simplex. Then $P$ is thin if and only if $P$ is a lattice pyramid.\label{cor:gorsimplex}
\end{cor}

\begin{proof}

Let $P$ be thin. If $P$ is $g$-thin, then $\deg(P)=\deg(g_P)=0$ as $P$ is a simplex. Hence, $P$ is a unimodular simplex, in particular, a lattice pyramid.

Otherwise, Theorem~\ref{thm:main}(iii) implies that there are faces $F$ and $G$ of $P$ such that $P=F \ast_{\Gor} G$ with $G$ $g$-thin. As $G$ is also a simplex, the previous consideration shows that $G$ is a unimodular simplex, thus, a lattice pyramid. Hence, Corollary~\ref{cor:pyr} implies that $P$ is also a lattice pyramid.
\end{proof}

In particular, if a Gorenstein simplex $P$ satisfies $\dim(P) \ge 2 \deg(P)$, then $P$ is a lattice pyramid. This statement can also be deduced from \cite[Cor.~3.10(2)]{Adjunction}.

\medskip
\bibliographystyle{alpha}
\bibliography{l-star_references}

\end{document}